\documentclass[english,reqno]{siamltex}
\usepackage[T1]{fontenc}
\usepackage[latin9]{inputenc}
\usepackage{prettyref}
\usepackage{float}
\usepackage{amstext}
\usepackage{amssymb}
\usepackage[dvips]{graphicx}
\usepackage{setspace}
\usepackage{xargs}[2008/03/08]
\makeatletter

\floatstyle{ruled}
\newfloat{algorithm}{tbp}{loa}
\providecommand{\algorithmname}{Algorithm}
\floatname{algorithm}{\protect\algorithmname}

\newcommand\eqref[1]{(\ref{#1})}
\newcommand{\code}[1]{\texttt{#1}}

\usepackage{etex}

\let\iint\@undefined
\let\iiint\@undefined
\let\iiiint\@undefined
\let\eqref\@undefined
\let\diag\@undefined
\let\*\@undefined

\usepackage{dimadefs}
\usepackage{enumerate}

\usepackage{fullpage}\usepackage{rotating}\usepackage{pstricks}\usepackage{tikz}

\usepackage{prettyref}
\newrefformat{sub}{Subsection \ref{#1}}
\newrefformat{subfig}{Subfigure \ref{#1}}
\newrefformat{app}{Appendix \ref{#1}}
\newrefformat{cor}{Corollary \ref{#1}}
\newrefformat{alg}{Algorithm \ref{#1}}
\newrefformat{step}{step \ref{#1}}
\newrefformat{def}{Definition \ref{#1}}
\newrefformat{clm}{Claim \ref{#1}}
\newrefformat{prop}{Proposition \ref{#1}}
\newrefformat{pblm}{Problem \ref{#1}}

\let\f\@undefined

\DeclareMathOperator{\shift}{E}
\DeclareMathOperator{\id}{I}
\DeclareMathOperator{\diag}{diag}
\DeclareMathOperator{\diam}{diam}
\DeclareMathOperator{\dist}{dist}

\newcommand{\ff}[2]{\ensuremath{(#1)_{#2}}} 
\newcommand{\newsomething}[1]{\widetilde{#1}}

\renewcommand{\vec}[1]{\ensuremath{{\boldsymbol #1}}}
\newcommand{\fun}{\ensuremath{f}}

\newcommand{\ivm}{\ensuremath{\mathcal{N}}} 
\newcommand{\dl}{\ensuremath{\Delta}}

\newcommand{\np}{\ensuremath{\mathcal{K}}} 
\newcommand{\ncoeffs}{\ensuremath{C}}
\newcommand{\nparams}{\ensuremath{R}}
\newcommand{\jp}{\ensuremath{\xi}}
\newcommand{\jc}{\ensuremath{a}}
\newcommand{\ord}{\ensuremath{d}}

\newcommand{\nn}[1]{\newsomething{#1}}

\usepackage{nameref}
\date{}

\usepackage{bookmark}

\@ifundefined{showcaptionsetup}{}{%
 \PassOptionsToPackage{caption=false}{subfig}}
\usepackage{subfig}
\makeatother

\usepackage{babel}

\begin{document}

\title{On the accuracy of solving confluent Prony systems\thanks{Department of Mathematics, Weizmann Institute of Science, Rehovot 76100, Israel.}}

\author{Dmitry Batenkov\thanks{\tt{dima.batenkov@weizmann.ac.il}. This author is supported by the Adams Fellowship Program of the Israel Academy of Sciences and Humanities.}
\and Yosef Yomdin\thanks{\tt{yosef.yomdin@weizmann.ac.il}. This author is supported by ISF grant 264/09 and the Minerva Foundation.}}
\maketitle
\begin{abstract}
In this paper we consider several nonlinear systems of algebraic equations
which can be called ``Prony-type''. These systems arise in various
reconstruction problems in several branches of theoretical and applied
mathematics, such as frequency estimation and nonlinear Fourier inversion.
Consequently, the question of stability of solution with respect to
errors in the right-hand side becomes critical for the success of
any particular application. We investigate the question of ``maximal
possible accuracy'' of solving Prony-type systems, putting stress
on the ``local'' behavior which approximates situations with low
absolute measurement error. The accuracy estimates are formulated
in very simple geometric terms, shedding some light on the structure
of the problem. Numerical tests suggest that ``global'' solution
techniques such as Prony's algorithm and ESPRIT method are suboptimal
when compared to this theoretical ``best local'' behavior.\end{abstract}

\section{Introduction}

\newcommandx\meas[1][usedefault, addprefix=\global, 1=k]{m_{#1}}
\global\long\def\nmeas{S}
\newcommandx\fwm[1][usedefault, addprefix=\global, 1=\nmeas]{\mathcal{P}_{#1}}
\newcommandx\man[2][usedefault, addprefix=\global, 1=\nparams, 2=\nmeas]{\mathcal{M}_{#1,#2}}
\global\long\def\noims{\tilde{\vec{y}}}
\global\long\def\exams{\vec{y}}
\global\long\def\src{\vec{x}}
\global\long\def\noisrc{\tilde{\src}}

\subsection{Problem definition}

Consider the following system of algebraic equations: 
\begin{align}
\sum_{i=1}^{\np}\jc_{i}\jp_{i}^{k} & =m_{k}\label{eq:basic-prony}
\end{align}
where $\jc_{i},\jp_{i}\in\complexfield$ are unknown parameters and
the measurements $\left\{ m_{k}\right\} _{k=0,1,\dotsc,}$ are given.
This ``exponential fitting'' system, or ``Prony system'', appears
in several branches of theoretical and applied mathematics, such as
frequency estimation, Padé approximation, array processing, statistics,
interpolation, quadrature, radar signal detection, error correction
codes, and many more. The literature on this subject is huge (for
instance, the bibliography on Prony's method from \cite{Auton1981}
is some 50+ pages long). Our interest in this system (and other, more
general systems of this kind, to be specified below) is motivated
by its central role in Algebraic Sampling -- a recent approach to
reconstruction of non-linear parametric models from measurements.
There, it arises as the problem of reconstructing a signal modeled
by a linear combination of Dirac $\delta$-distributions: 
\begin{align}
f(x) & =\sum_{i=1}^{\np}\jc_{i}\delta(x-\jp_{i}),\quad\jc_{i},\jp_{i}\in\reals\label{eq:delta-fun}
\end{align}
 from the measurements given by the power moments 
\begin{align}
\meas(f)\isdef\int_{0}^{1}x^{k}f(x)\dd x.\label{eq:moments-def}
\end{align}

While the above problem may be considered mainly of theoretical interest,
it is actually one of the most basic ones in Algebraic Sampling. On
one hand, if $s(x)$ is a piecewise-constant signal with jump discontinuities
at the locations $\jp_{1},\dotsc,\jp_{\np}$, then $s'(x)=f(x)$ as
in \eqref{eq:delta-fun}. Thus, the ``signal'' $f(x)$ essentially
captures the non-smooth nature of $s(x)$. On the other hand, the
moments \eqref{eq:moments-def} are convenient to consider because
of the respective simplicity of the arising algebraic equations, while
other types of measurements (e.g. Fourier coefficients) may be recast
into moments after change of variables.

An important generalization of the Prony system, which is of great
interest to us, arises when the simple model \eqref{eq:delta-fun}
is extended to include higher-order derivatives (see \cite{bat2008,vetterli2002ssf}
for examples of such constructions): 
\begin{align}
f(x) & =\sum_{i=1}^{\np}\sum_{j=0}^{l_{i}-1}\jc_{ij}\der{\delta}{j}(x-\jp_{i}),\quad\jc_{i,j},\jp_{j}\in\reals\label{eq:gen-delta-fun}
\end{align}
where $\der\delta j$ is the $j$-th derivative of the Dirac delta
(in the sense of distributions).

From now on, we denote the number of unknown coefficients $\jc_{i,j}$
by $\ncoeffs\isdef\sum_{i=1}^{\np}l_{i}$, and the overall number
of unknown parameters by $\nparams\isdef\ncoeffs+\np$. Taking moments
of $\fun\left(x\right)$ in \eqref{eq:gen-delta-fun}, we arrive%
\footnote{Strictly speaking, this will result in a ``real'' confluent Prony
system.%
} at the following ``confluent Prony'' system: 
\begin{equation}
\sum_{i=1}^{\np}\sum_{j=0}^{l_{i}-1}\jc_{i,j}\ff{k}{j}\jp_{i}^{k-j}=\meas\qquad\jc_{ij},\jp_{i},\meas\in\complexfield\label{eq:gen-prony}
\end{equation}
where the Pochhammer symbol $\ff{i}{j}$ denotes the falling factorial
\[
\ff{i}{j}=i(i-1)\cdot\dotsc\cdot(i-j+1),\qquad i\in\reals,\; j\in\naturals
\]
and the expression $\ff kj\jp_{i}^{k-j}$ is defined to be zero for
$k>j$.

The Prony-type systems appear in various recent reconstruction methods
of signals with discontinuities - see \cite{banerjee1998exponentially,bat2008,batGolubYom2010,BatGolYom2011,batyomAlgFourier,beckermann2008rgp,candes2012towards,driscoll2001pba,eckhoff1995arf,golub2000snm,gustafsson2000rpd,kvernadze2004ajd,march1998apm,mayergoiz94}.
In particular, Finite Rate of Innovation (FRI) techniques \cite{dragotti2007sma,maravic2005sar,vetterli2002ssf}
have spawned a rather extensive literature (see e.g. a recent addition
\cite{uriguenx2011exponential}). Usually, the $\jp_{i}$ represent
``location'' parameters of the problem, such as discontinuity locations
or complex frequencies $\jp_{j}=\ee^{\imath\omega_{j}}$. These variables
enter the equations in a nonlinear way, and we call them ``nodes''.
The coefficients $\jc_{ij}$, on the other hand, enter the equations
linearly, and we call them ``magnitudes''. 

While Algebraic Sampling provides exact reconstruction for noise-free
data in many cases mentioned above, a critical issue remains - namely,
stability, or accuracy of solution. Stable solution of Prony-type
systems is generally considered to be a difficult problem, and in
recent years many algorithms have been devised for this task (e.g.
\cite{badeau2008performance,Holmstrom200231,kahn1992cps,osborne1975some,peter2011nonlinear,potts2010parameter,rao1992mbp,stoica1989music,vanblaricum1978pas}).
Perhaps the simplest version of the stability problem can be formulated
as follows (cf. \prettyref{def:pt-acc}, \prettyref{def:local-pt-acc}
and \prettyref{sub:summary-of-results}).

\emph{\label{pblm:stability-pblm}Assume that the measurements $\left\{ \meas\right\} _{k=0,\dots,\nmeas-1}$
are known with some error: $\meas+\varepsilon_{k}$.} \emph{Given
an estimate $\varepsilon=\max_{k}|\varepsilon_{k}|$, how large can
the error in the reconstructed model parameters (i.e. $\left|\Delta\jp_{j}\right|\isdef|\nn{\jp_{j}}-\jp_{j}|$
and $\left|\Delta\jc_{i,j}\right|\isdef|\nn{\jc_{i,j}}-\jc_{i,j}|$)
be in the worst case in terms of $\epsilon$, number of measurements
$\nmeas$ and the true parameters $\left\{ \jp_{j}\right\} ,\left\{ \jc_{i,j}\right\} $?}

In more detail, our ultimate goal may be described as follows:
\begin{enumerate}
\item determine the qualitative dependence of the accuracy on the values
of the parameters;
\item quantify this dependence as precisely as possible;
\item determine how (and if at all) increasing the number of measurements
(i.e. oversampling) improves accuracy.
\end{enumerate}

\subsection{\label{sub:related-work}Related work}

Matching the ubiquity of Prony-type systems is the impressive body
of literature devoted to both designing methods of solution and analyzing
the accuracy/robustness of these methods, see references above. Although
there appears to be no simple answer to the above question of ``maximal
possible accuracy'', several important results in this direction
are available in the literature, which we now briefly discuss.

Methods of solution can be roughly divided into three categories (see
e.g. \cite{so2003new},\cite[Section 4]{stoica2005spectral}): direct
nonlinear minimization (nonlinear least squares), recurrence-based
methods (such as original Prony's method - see \prettyref{sec:prony-method})
and subspace methods (such as Pisarenko's method, MUSIC, ESPRIT, matrix
pencils - see e.g. \cite{rao1992mbp}).

In the framework of statistical signal estimation \cite{kay1993fundamentals},
the subspace methods are known to be more efficient and robust to
noise, mainly due to the fact that the noise is assumed to have certain
statistical properties. The confluent Prony system \eqref{eq:gen-prony}
is also known as ``polynomial amplitude complex exponential'' (PACE)
model. A standard measure of estimator performance is Cramer-Rao (and
related) lower bounds (CRB). These have been recently established
for the PACE model in \cite{badeau2008cramer} (see also related results
for FRI models \cite{Ben-Haim2010}). Furthermore, it has been demonstrated
that the performance of the generalized ESPRIT algorithm (\cite{badeau2006high,badeau2008performance}
and \prettyref{sub:ESPRIT}) is close to the optimal CRB, therefore
we consider it to represent the state of the art in the subspace methods.

We do not assume any particular statistical model or other structure
for either the error terms $\varepsilon_{k}$ or the estimation algorithm
(such as white noise or unbiasedness). Therefore, the CRB and related
lower bounds cannot provide the full answer to the stability problem
\emph{as is}. Still, it turns out that the stability bounds developed
in this paper resemble the CRB as established in \cite{badeau2008cramer},
see \prettyref{sub:crb-pace} below for details.

Recent papers of Tasche et al. \cite{peter2011nonlinear,potts2010parameter}
contain some uniform error bounds for solving Prony systems. In particular,
the authors develop the so-called Approximate Prony method, analyze
its worst-case error and numerically compare it with the ESPRIT method
(showing similar performance). Although they consider the non-confluent
version of the Prony system \eqref{eq:basic-prony} and analyze only
the error in recovering the magnitudes $\jc_{j}$, we believe these
results to be an important step towards answering the stability problem
as posed above. See \prettyref{sub:apm} below for details.

Very recently, Candes et al. \cite{candes2012towards} investigated
stable solution of Prony systems by total variation minimization under
assumptions of minimal node separation, in the context of super-resolution.

Considering all the above, we believe that a full answer to our somewhat
rigid $l^{\infty}$ formulation of the stability problem may contribute
to the understanding of limitations of using Prony systems and methods
both in signal processing applications and in function approximation,
in particular compressed sensing, nonlinear Fourier inversion, Finite
Rate of Innovation techniques and related problems.

\subsection{Notation}

In the sequel we use the infinity norm distance
\[
\forall\vec{x},\vec{y}\in\complexfield^{n}:\qquad\dist\left(\vec{x},\vec{y}\right)\isdef\max_{1\leq i\leq n}\left|x_{i}-y_{i}\right|,
\]
and denote by $B\left(\vec{a},\varepsilon\right)$ the $\varepsilon$-ball
around a point $\vec{a}\in\complexfield^{n}$ in this norm.

\subsection{Summary of results\label{sub:summary-of-results}}

In \prettyref{sec:prony-map} we define ``best possible point-wise
accuracy'' as follows. We consider the ``Prony map'' $\fwm:\complexfield^{\nparams}\to\complexfield^{\nmeas}$
which associates to any parameter vector $\src=\bigl\{\{\jc_{ij}\},\{\jp_{i}\}\bigr\}\in\complexfield^{\nparams}$
its corresponding measurement vector $\exams=\left(\meas[0],\dots,\meas[\nmeas-1]\right)\in\complexfield^{\nmeas}$
(where the $\meas$ are given by \eqref{eq:gen-prony}). 

Now if instead of $\exams$ we are given a noisy $\noims\in B\left(\exams,\varepsilon\right)$,
then this $\noims$ can correspond to any parameter vector $\noisrc\in\complexfield^{\nparams}$
for which $\fwm\left(\noisrc\right)\in B\left(\noims,\varepsilon\right)$.
Therefore we define the best possible accuracy at a point $\src$
to be equal to the maximal (over all $\noims)$ spread of the preimage
of this $B\left(\noims,\varepsilon\right)$, that is (see \prettyref{def:pt-acc})
\[
\sup_{\noims\in B\left(\exams,\varepsilon\right)}\frac{1}{2}\diam\fwm^{-1}\left(B\left(\noims,\varepsilon\right)\right).
\]

We then simplify the setting by assuming that the number of measurements
$\nmeas$ equals the number of unknowns $\nparams$, and looking at
the (local) linear approximation to the Prony map $\fwm$. Then the
solution error at some (non-critical) point in the parameter space
can be estimated by the local Lipschitz constant of the (regular)
inverse map $\fwm^{-1}$. We derive such simple estimates in \prettyref{sec:local-accuracy},
and compare them to the ``global'' accuracy of the original Prony
method (derived for completeness in \prettyref{sec:prony-method}).

Our main result (\prettyref{thm:local-lipshitz-estimates}) can be
summarized as follows (all statements are valid for small $\varepsilon$):
\begin{enumerate}
\item The stability of recovering a node $\jp_{i}$ depends on the separation
of the nodes and is inversely proportional to the magnitude of the
highest coefficient corresponding to this node ($\left|\jc_{i.l_{i}-1}\right|$),
and does not depend on any other magnitude.
\item For $1\leq j\leq l_{i}-1$, the stability of recovering a magnitude
$\jc_{i,j}$ depends on the separation of the nodes, is proportional
to $1+\frac{\left|\jc_{i,j-1}\right|}{\left|\jc_{i,l_{i}-1}\right|}$,
and does not depend on any other magnitude. Note that in fact every
magnitude influences only the next highest magnitude corresponding
to the same node.
\item The stability of recovering the lowest magnitudes $\jc_{i,0}$ is
the same for all nodes and it depends only on the separation of the
nodes.
\end{enumerate}
The separation of the nodes is specified in terms of norms of inverse
confluent Vandermonde matrices on the nodes, which is roughly of the
same order as some finite power of $\prod_{1\leq i<j\leq\np}\left|\jp_{j}-\jp_{i}\right|^{-1}$.

Our numerical experiments (\prettyref{sec:experiments}) confirm the
above theoretical estimates. We also test the performance of two well-known
solution methods - namely the recurrence-based Prony method (\prettyref{sec:prony-method})
and the generalized ESPRIT (\prettyref{sub:ESPRIT}) - in the same
setting as above (i.e. high SNR). The results suggest that:
\begin{enumerate}
\item The recurrence-based global Prony method does not achieve the above
theoretical limits, and so it is not optimal even in the case of small
data perturbations.
\item The subspace methods (in particular the ESPRIT algorithm) behave better
than the Prony method but still they are not optimal for small perturbations
and small sample size.
\end{enumerate}
The ``Prony map'' approach can in principle be generalized to obtain
both global accuracy bounds as well as study effects of oversampling
- by considering the case $\nmeas>\nparams$ and taking into account
second-order terms in the Taylor expansion of $\fwm$. We discuss
these directions in \prettyref{sec:discussion}.

\subsection{Acknowledgments}

We are grateful to the two anonymous referees, whose comments, suggestions
and references were very helpful.

\section{The Prony method}

\label{sec:prony-method}

In this section we describe the most basic solution method for the
system \eqref{eq:gen-prony}, which is in fact a slight generalization
of the (historically earliest) method due to Prony \cite{prony1795essai}.
By factorizing the so-called ``data matrix'', one immediately obtains
necessary and sufficient conditions for a unique solution, as well
as an estimate of the numerical stability of the method.

Most of the results of \prettyref{sec:prony-method} are not new and
are scattered throughout the literature. Nevertheless, we believe
that our presentation can be useful for further study of the various
singular situations, such as collision of two nodes.

\subsection{The description of the method}

The non-trivial part is the recovery of the nodes $\jp_{j}$. Note
that the case of a-priori known nodes has been extensively treated
in the literature (see e.g. \cite{Adcock2010,Poghosyan2010} for the
most recent results). Using the framework of finite difference calculus,
one can easily prove the following result (see \cite[Theorem 2.8]{bat2008}).
\begin{proposition}
Let the sequence $\{\meas\}$ be given by \eqref{eq:gen-prony}. Then
this sequence satisfies the recurrence relation (of length at most
$\ncoeffs+1$) 
\begin{align*}
\biggl(\prod_{i=1}^{\np}(\shift-\jp_{i}\id)^{l_{i}}\biggr)\{\meas[k]\} & =0
\end{align*}
 where $\shift$ is the forward shift operator in $k$ and $\id$
is the identity operator.\end{proposition}
\begin{corollary}
For all $k\in\naturals$ we have the recurrence relation $\sum_{j=0}^{\ncoeffs}q_{j}\meas[k+j]=0$
where $q_{0},q_{1},\dots,q_{\ncoeffs}$ are the coefficients of the
polynomial $q(x)\isdef\prod_{i=1}^{\np}(x-\jp_{i})^{l_{i}}$. 
\end{corollary}
\begin{minipage}[t]{1\columnwidth}%
\end{minipage}

This suggests the following reconstruction procedure%
\footnote{Equivalent derivation of the method is based on Padé approximation
to the function $I(z)=\sum_{k=0}^{\infty}\meas[k]z^{k}$ -- see \cite{prony1795essai}
and, for instance, \cite{sig_ack}.%
}.

\begin{algorithm}[h]
Let there be given $\{\meas[k]\}_{k=0}^{2\ncoeffs-1}$ (where $\ncoeffs=\sum_{i=1}^{\np}l_{i}$). 
\begin{enumerate}
\item \label{step:prony-system-for-polynomial}Solve the linear system (here
we set $q_{\ncoeffs}=1$ for normalization) {\footnotesize 
\begin{align}
\underbrace{\begin{pmatrix}\meas[0] & \meas[1] & \cdots & \meas[\ncoeffs-1]\\
\meas[1] & \meas[2] & \cdots & \meas[\ncoeffs]\\
\vdots & \vdots & \vdots & \vdots\\
\meas[\ncoeffs-1] & \meas[\ncoeffs] & \cdots & \meas[2\ncoeffs-2]
\end{pmatrix}}_{\isdef M_{\ncoeffs}}\begin{pmatrix}q_{0}\\
q_{1}\\
\vdots\\
q_{\ncoeffs-1}
\end{pmatrix}=-\begin{pmatrix}\meas[\ncoeffs]\\
\meas[\ncoeffs+1]\\
\vdots\\
\meas[2\ncoeffs-1]
\end{pmatrix}\label{eq:poly-meas-hankel-system}
\end{align}
} for the unknown coefficients $q_{0},\dotsc,q_{\ncoeffs-1}$. 
\item \label{step:prony-roots}Find all the roots of $q(x)=x^{\ncoeffs}+\sum_{j=0}^{\ncoeffs-1}q_{i}x^{i}$.
These roots, with appropriate multiplicities, are the unknowns $\jp_{1},\dotsc,\jp_{\np}$
(use e.g. arithmetic means to estimate multiple roots which are scattered
by the noise into clusters).
\item \label{step:linear-system-for-magnitudes}Substitute the recovered
$\jp_{i}$'s back into the original equations \eqref{eq:gen-prony}.
Solve the resulting overdetermined linear system ($\ncoeffs$ unknowns
and $2\ncoeffs$ equations) with respect to the magnitudes $\left\{ \jc_{i,j}\right\} $
by least squares method.
\end{enumerate}
\caption{The Prony method}
\label{alg:prony-method}
\end{algorithm}

Several comments are in order.
\begin{enumerate}
\item The number of measurements used in \prettyref{step:prony-system-for-polynomial}
equals $2\ncoeffs$ which can be greater than the number of unknowns
$\nparams=\ncoeffs+\np$ (equality for order zero Prony system). If
more measurements are available, the linear system \eqref{eq:poly-meas-hankel-system}
can be modified in a straightforward way to be overdetermined, and
subsequently solved by, say, the least squares method.
\item The linear system for the magnitudes has a special ``Vandermonde''-like
structure (see below), and so certain efficient algorithms can be
used to solve it (e.g. \cite{bjorck1973acv,Lu:1994:FSC:196045.196078}). 
\end{enumerate}
The remainder of this section is organized as follows. The Hankel
matrix $M_{\ncoeffs}$ is shown to factor into the product of a generalized
``Vandermonde-type'' matrix which depends only on the nodes $\jp_{j}$,
with a upper triangular matrix depending only on the amplitudes $\jc_{i,j}$.
We also write down explicitly the linear system for the $\jc_{i,j}$
(see \prettyref{step:linear-system-for-magnitudes} in \prettyref{alg:prony-method}
above). These calculations lead to simple non-degeneracy conditions
and stability estimates for the Prony method.

\newcommandx\confvec[2][usedefault, addprefix=\global, 1=j, 2=k]{\vec{u_{#1,#2}}}
 \global\long\def\cvand{\ensuremath{U}}

\subsection{Factorization of the data matrix}

Let us start by recalling a well-known type of matrices.
\begin{definition}
For every $j=1,\dotsc,\np$ and $k\in\naturals$ let the symbol $\confvec$
denote the following $1\times l_{j}$ row vector
\begin{equation}
\confvec\isdef\left[\begin{array}{cccc}
\jp_{j}^{k}, & k\jp_{j}^{k-1}, & \dots & ,\ff{k}{l_{j}-1}\jp_{j}^{k-l_{j}+1}\end{array}\right].\label{eq:confvec-def}
\end{equation}

\end{definition}
\begin{minipage}[t]{1\columnwidth}%
\end{minipage} 
\begin{definition}
Let $\cvand=\cvand(\jp_{1},l_{1},\dotsc,\jp_{\np},l_{\np})$ denote
the matrix
\begin{equation}
\cvand=\left[\begin{array}{cccc}
\confvec[1][0] & \confvec[2][0] & \dotsc & \confvec[\np][0]\\
\confvec[1][1] & \confvec[2][1] & \dotsc & \confvec[\np][1]\\
 &  & \dotsc\\
\confvec[1][\ncoeffs-1] & \confvec[2][\ncoeffs-1] & \dotsc & \confvec[\np][\ncoeffs-1]
\end{array}\right].\label{eq:confluent-vandermonde-def}
\end{equation}

\end{definition}
This matrix is called the ``confluent Vandermonde'' (\cite{bjorck1973acv,gautschi1962iva})
matrix. It has been long known in numerical analysis due to its central
role in Hermite polynomial interpolation. Its determinant is (\cite[p.30]{schumaker1981sfb})
\begin{equation}
\det\cvand=\prod_{1\leq i<j\leq\np}(\jp_{j}-\jp_{i})^{l_{j}l_{i}}\prod_{\mu=1}^{\np}\prod_{\nu=1}^{l_{\mu}-1}\nu!.\label{eq:cvand-det}
\end{equation}

It is straightforward to see that the matrix $\cvand$ defines the
linear system for the jump magnitudes $\jc_{i,j}$.
\begin{proposition}
Let $\vec{a}$ be the column vector containing all the magnitudes
$\left\{ \jc_{i,j}\right\} $, i.e.
\[
\vec{\jc}\isdef[\jc_{1,0},\dots,\jc_{1,l_{1}-1},\jc_{2,0},\dots,\jc_{2,l_{2}-1},\dots,\jc_{\np,0},\jc_{\np,l_{\np}-1}]^{T}
\]
and $\vec{{m}}\isdef[\meas[0],\dots,\meas[\ncoeffs-1]]^{T}$. Then
we have 
\begin{equation}
\cvand(\jp_{1},l_{1},\dotsc,\jp_{\np},l_{\np})\vec{\jc}=\vec{{m}.}\label{eq:conf-prony-linear-system}
\end{equation}

\end{proposition}
It is known that every Hankel matrix $H$ admits a factorization $H=\cvand D\cvand^{T}$,
where $\cvand$ is given by \eqref{eq:confluent-vandermonde-def}
and $D$ is a block diagonal matrix -- see \cite{boley1998vfh}. Using
different notations, such a factorization is proved in \cite[Proposition III.7]{badeau2006high}
for the Hankel matrix $M_{\ncoeffs}$.
\begin{lemma}
\label{lem:conf-prony-fact} For the system \eqref{eq:gen-prony},
the matrix $M_{\ncoeffs}$ admits the following factorization: 
\begin{equation}
M_{\ncoeffs}=\cvand B\cvand^{T}\label{eq:prony-hankel-main-factorization}
\end{equation}
 where $\cvand=\cvand(\jp_{1},l_{1},\dotsc,\jp_{\np},l_{\np})$ is
the confluent Vandermonde matrix \eqref{eq:confluent-vandermonde-def}
and $B$ is the $\ncoeffs\times\ncoeffs$ block diagonal matrix $B=\diag\{B_{1},\dotsc,B_{\np}\}$
with each block of size $l_{i}\times l_{i}$ given by {\small 
\begin{equation}
B_{i}\isdef\begin{bmatrix}\jc_{i0} & \jc_{i1} & \cdots & \cdots & \jc_{i,l_{i}-1}\\
\jc_{i1} &  &  & {l_{i}-1 \choose l_{i}-2}\jc_{i,l_{i}-1} & 0\\
\cdots &  &  & \cdots & 0\\
 & {l_{i}-1 \choose 2}\jc_{i,l_{i}-1} & 0 & \cdots & 0\\
\jc_{i,l_{i}-1} & 0 & \cdots & \cdots & 0
\end{bmatrix}.\label{eq:b-def}
\end{equation}
} In other words, $B_{i}$ is a ``flipped'' upper triangular matrix
whose $j$-th anti-diagonal equals to 
\[
\jc_{ij}\cdot\begin{bmatrix}1 & {j \choose 2} & \cdots & \ {j \choose j-1} & 1\end{bmatrix}
\]
 for $j=0,\dotsc,l_{i}-1$.
\end{lemma}
\begin{minipage}[t]{1\columnwidth}%
\end{minipage}

The formula \eqref{eq:prony-hankel-main-factorization} is useful
because it separates the jump locations $\{\jp_{i}\}$ from the magnitudes
$\{\jc_{i,j}\}$, simplifying the analysis considerably.
\begin{theorem}
\label{thm:uniqueness-conf-prony} The system \eqref{eq:gen-prony}
for $k=0,1,\dots,2\ncoeffs$ has a unique solution if and only if
all the $\left\{ \jp_{i}\right\} $'s are pairwise different and all
the $\left\{ \jc_{i,l_{i}-1}\right\} $'s (just the highest coefficients)
are nonzero.\end{theorem}
\begin{proof}
Existence of a unique solution to the system \eqref{eq:poly-meas-hankel-system}
is equivalent to the non-degeneracy of $M_{\ncoeffs}=\cvand B\cvand^{T}$.
Furthermore, the system for the jump magnitudes is given by \eqref{eq:conf-prony-linear-system}.
Therefore, existence of a unique solution to \eqref{eq:gen-prony}
is equivalent to the conditions $\det\cvand\neq0$ and $\det B\neq0$.
The proof is completed by \eqref{eq:cvand-det} and \eqref{eq:b-def}. 
\end{proof}

\subsection{Stability estimates\label{sub:prony-method-stability}}

The stability of the Prony method can be estimated by the condition
numbers of the matrices $B$ and $\cvand$. In particular, we have
the following well-known result (e.g. \cite{wilkinson1994rounding})
from numerical linear algebra.
\begin{lemma}
\label{lem:condition-number}Consider the linear system $A\vec{x}=\vec{b}$
and let $\vec{x_{0}}$ be the exact solution. Let this system be perturbed:
\[
\left(A+\Delta A\right)\vec{x}=\vec{b}+\vec{\Delta b}
\]
 and let $\vec{x_{0}}+\vec{\Delta x}$ denote the exact solution of
this perturbed system. Denote $\delta x=\frac{\|\vec{\Delta x}\|}{\|\vec{x_{0}}\|},\delta A=\frac{\|\Delta A\|}{\|A\|},\delta b=\frac{\|\vec{\Delta b}\|}{\|\vec{b}\|}$
and the condition number $\kappa=\|A\|\|A^{-1}\|$ for some vector
norm $\|\cdot\|$ and the induced matrix norm. Then
\begin{equation}
\delta x\leq\frac{\kappa}{1-\kappa\cdot\delta A}\left(\delta A+\delta b\right).\label{eq:accuracy-through-condition-number}
\end{equation}

\end{lemma}
Now we can easily estimate the stability of the Prony method (compare
with similar estimates in \cite[eq. (19)]{badeau2006high}).
\begin{corollary}
\label{cor:prony-stability} Let the measurements $\left\{ m_{k}\right\} $
be given with an error bounded by $\varepsilon$. Denote $u=\kappa(\cvand),b=\kappa(B)$.
Assume that $\left|\jp_{i}\right|\leq\Xi$ for all $i=1,\dots,\np$.
Then the Prony method recovers the parameters $\{\jp_{j},\jc_{i,j}\}$
with the following accuracy as $\varepsilon\to0$: 
\begin{align*}
|\Delta\jp_{j}| & \sim\left(u^{2}b\varepsilon\right)^{\frac{1}{l_{j}}}+O\left(\varepsilon^{\frac{2}{l_{j}}}\right)\\
|\Delta\jc_{i,j}| & \sim C\left(\Xi\right)u\left(u^{2}b\varepsilon\right)^{\frac{1}{\max_{j}l_{j}}}+\text{L.O.T.}
\end{align*}
where $C\left(\Xi\right)$ is a constant depending on the number $\Xi$.\end{corollary}
\begin{proof}
Using the factorization of \prettyref{lem:conf-prony-fact}, we obtain
that $\kappa\left(M_{\ncoeffs}\right)\leq u^{2}b$. Therefore, according
to \eqref{eq:accuracy-through-condition-number} the coefficient vector
$\vec{q}=\left(q_{0},\dots,q_{\ncoeffs-1}\right)$ is recovered with
the accuracy 
\begin{align*}
\|\delta\vec{q}\| & \sim\frac{\kappa\left(M_{\ncoeffs}\right)}{1-\kappa\left(M_{\ncoeffs}\right)\delta M_{\ncoeffs}}\cdot\bigl(\delta M_{\ncoeffs}+\delta\vec{m}\bigr)\\
 & \leq\frac{u^{2}b\varepsilon}{1-u^{2}b\varepsilon}\sim u^{2}b\varepsilon+O(\varepsilon^{2}).
\end{align*}
The parameters $\jp_{1},\dots,\jp_{\np}$ are the roots of the polynomial
with coefficient vector $\vec{q}$, with multiplicities $l_{1},\dots,l_{\np}$.
Therefore, by the general theory of stability of polynomial roots
(see e.g. \cite{wilkinson1994rounding}) it is known that $\Delta\jp_{j}\sim\left(\delta\vec{q}\right)^{\frac{1}{l_{j}}}$.
The first part of the claim is thus proved.

Now consider the linear system \eqref{eq:conf-prony-linear-system}
for recovering the jump magnitudes. Note that the matrix $\cvand$
is known only approximately. Again, by \eqref{eq:accuracy-through-condition-number}
we have 
\begin{align}
\begin{split}\delta\vec{a} & \sim\frac{\kappa\left(\cvand\right)}{1-\kappa\left(\cvand\right)\delta\cvand}\left(\delta\cvand+\delta\vec{m}\right)\end{split}
\label{eq:coeff-error-bound}
\end{align}
 Assuming that $\left|\jp_{j}\right|\leq\Xi$, it is easy to see that
$\delta\cvand\sim C\left(\Xi\right)\left(u^{2}b\varepsilon\right)^{\frac{1}{\max_{j}l_{j}}}$.
Plugging this value into \eqref{eq:coeff-error-bound} we get the
desired result. 
\end{proof}
\begin{minipage}[t]{1\columnwidth}%
\end{minipage}

Inverses of confluent Vandermonde matrices and their condition numbers
are extensively studied in numerical linear algebra (e.g. \cite{bazan2000crv,beckermann2000condition,gautschi1962iva})%
\footnote{In particular, the paper \cite[Theorem 3]{gautschi1962iva} contains
the following estimate for the norm of $\{\cvand\left(\jp_{1},1,\dots,\jp_{\np},1\right)\}^{-1}$
when the nodes are arbitrary complex numbers: 
\[
\|U^{-1}\|_{\infty}\leq\max_{1\leq i\leq\np}b_{i}\prod_{j=1,j\neq i}^{\np}\biggl(\frac{1+|\jp_{j}|}{|\jp_{i}-\jp_{j}|}\biggr)^{2}
\]
 where 
\[
b_{i}\isdef\max\biggr(1+|\jp_{i}|,1+2(1+|\jp_{i}|)\sum_{j\neq i}\frac{1}{|\jp_{j}-\jp_{i}|}\biggl).
\]
}. In general, $\kappa(\cvand)$ will grow exponentially with $\np$
and will also depend on the ``node separation'' $\prod_{i\neq j}|\jp_{j}-\jp_{j}|^{-1}$.
As for $\kappa(B)$, we are not aware of a general formula except
for the simplest cases%
\footnote{The following are estimates of the spectral condition numbers. 
\begin{itemize}
\item For the standard Prony system we have 
\[
\kappa\bigl(B\bigr)=\frac{\max_{j}|\jc_{j,0}|}{\min_{j}|\jc_{j,0}|}.
\]
 
\item For multiplicity 1 confluent system, assuming $\jc_{j,1}\neq0$ and
denoting $\mu_{j}\isdef\frac{\jc_{j,0}}{\jc_{j,1}}$, brute force
calculation gives 
\[
\kappa\bigl(B\bigr)=\frac{\max_{j}\sqrt{\frac{\mu_{j}^{2}+2+\mu_{j}\sqrt{\mu_{j}^{2}+4}}{\mu_{j}^{2}+2-\mu_{j}\sqrt{\mu_{j}^{2}+4}}}}{\min_{j}\sqrt{\frac{\mu_{j}^{2}+2+\mu_{j}\sqrt{\mu_{j}^{2}+4}}{\mu_{j}^{2}+2-\mu_{j}\sqrt{\mu_{j}^{2}+4}}}}.
\]
\end{itemize}
}.

Finally, notice that the stability estimates of \prettyref{cor:prony-stability}
suggest that when the Prony method is used, the parameters of the
problem are ``coupled'' to each other, in the sense that the accuracy
of recovering either a node $\jp_{i}$ or a magnitude $\jc_{i,j}$
will depend on the values of \emph{all the parameters at once}. This
undesired behavior is confirmed by our numerical experiments in \prettyref{sec:experiments}.

\section{Measurement set and the Prony map\label{sec:prony-map}}

\global\long\def\noims{\tilde{\vec{y}}}
\global\long\def\exams{\vec{y}}
\global\long\def\src{\vec{x}}

Assume that the number of measurements is $\nmeas\geq\nparams$ (where
$\nparams$ is the overall number of parameters in the confluent Prony
system). Then we define $\man$ to be the set%
\footnote{Formally, $\man$ is a projection of the complex algebraic variety
defined by the set of the $\nmeas$ confluent Prony equations onto
the corresponding $\nmeas$ coordinate axes. If all parameters are
real-valued, this is a semialgebraic set.%
} of all possible exact measurements, i.e.
\[
\man\isdef\left\{ \left(\meas[0],\meas[1],\dots,\meas[\nmeas-1]\right):\quad\meas=\sum_{i=1}^{\np}\sum_{j=0}^{l_{i}-1}\jc_{i,j}\ff{k}{j}\jp_{i}^{k-j},\;\jc_{i,j}\in\complexfield,\;\jp_{j}\in\complexfield\right\} \subset\complexfield^{\nmeas}.
\]

This $\man$ is the image of $\complexfield^{\nparams}$ under the
``Prony map'' $\fwm:\complexfield^{\nparams}\to\complexfield^{\nmeas}$
defined as 
\begin{equation}
\fwm\left(\{\jc_{ij}\},\{\jp_{i}\}\right)=\left(\meas[0],\meas[1],\dots,\meas[\nmeas-1]\right):\quad\meas=\sum_{i=1}^{\np}\sum_{j=0}^{l_{i}-1}\jc_{i,j}\ff{k}{j}\jp_{i}^{k-j}.\label{eq:prony-map-def}
\end{equation}

Now let $\src=\bigl\{\{\jc_{ij}\},\{\jp_{i}\}\bigr\}\in\complexfield^{\nparams}$
be an unknown parameter vector and $\vec{\exams=\fwm\left(\src\right)\in\man}$
its corresponding exact measurement vector. The absolute error in
each measurement is bounded from above by $\varepsilon$, therefore
the actual measurement satisfies $\noims\in B\left(\exams,\varepsilon\right)$.
Now consider the set
\[
T_{\noims,\varepsilon}\isdef\man\cap B\left(\noims,\varepsilon\right)
\]
of all possible noise-free measurements corresponding to the given
noisy one $\noims$. Any algorithm which receives this $\noims$ as
input will therefore produce worst-case error which is at least
\[
\frac{1}{2}\diam\fwm^{-1}\left(T_{\noims,\varepsilon}\right)
\]
where $\fwm^{-1}$ denotes the full preimage set.

This prompts us to make the following definition.

\global\long\def\accr{\mathcal{ACC}}

\begin{definition}
\label{def:pt-acc}Assign to each one of the parameters $\{\jc_{ij}\},\{\jp_{i}\}$
a unique index $1\leq p\leq\nparams$. The best possible \emph{point-wise
accuracy }of solving the noisy confluent Prony system \eqref{eq:gen-prony}
with each noise component bounded above by $\varepsilon$ at the point
$\src=\left(\{\jc_{ij}\},\{\jp_{i}\}\right)\in\complexfield^{\nparams}$
with respect to the parameter $p$ is defined to be
\[
\accr\left(\src,\varepsilon,p\right)\isdef\sup_{\noims\in B\left(\fwm\left(\src\right),\varepsilon\right)}\frac{1}{2}\diam_{p}\fwm^{-1}\left(\man\cap B\left(\noims,\varepsilon\right)\right)
\]
where $\diam_{p}A$ is the diameter of the set $A$ along the dimension
$p$.
\end{definition}
\begin{minipage}[t]{1\columnwidth}%
\end{minipage}

Obviously, $\accr\left(\src,\varepsilon\right)$ will depend on the
point $\src\in\complexfield^{\nparams}$ in a nontrivial way because
the chart $\fwm$ is nonlinear. Calculation of the function $\accr$
may be considered as one possible answer to the stability problem
posed in the Introduction.

\section{Local accuracy\label{sec:local-accuracy}}

\global\long\def\jac{\ensuremath{\mathcal{J}}}
 \global\long\def\laccr{\accr_{LOC}}

Having given the general definition of accuracy, in the remainder
of this paper we restrict ourselves to the ``local'' setting in the
following sense: we assume that $\varepsilon$ is small enough so
that the set $\man$ can be approximated by the linear part of the
Prony map, and furthermore we take $\nmeas=\nparams$ so that the
preimage will be given by the usual inverse function. For such an
analysis to be valid, it should be done at non-critical points of
$\fwm$ so that this map is locally invertible. By definition, the
point $\vec{x}$ is a critical point of $\fwm$ if the Jacobian determinant
of $\fwm$ vanishes at $\vec{x}$.

To summarize, let us give the following definition of the local accuracy
which is nothing more than the first-order Taylor approximation to
the inverse function $\ivm=\fwm^{-1}$ at a regular point of $\fwm$. 
\begin{definition}
\label{def:local-pt-acc}Assume $\nmeas=\nparams$. Let $\src=\left(\{\jc_{ij}\},\{\jp_{i}\}\right)\in\complexfield^{\nparams}$
be a regular point of $\fwm$ and assume $\varepsilon$ to be small
enough so that that the inverse function $\ivm=\fwm^{-1}$ exists
in $\varepsilon$-neighborhood of $\exams=\fwm\left(\src\right)$.
Assign, as before, to each one of the parameters $\{\jc_{ij}\},\{\jp_{i}\}$
a unique index $1\leq p\leq\nparams$. The best possible \emph{local
point-wise accuracy }of solving the noisy confluent Prony system \eqref{eq:gen-prony}
with each noise component bounded above by $\varepsilon$ at the point
$\src$ with respect to the parameter $p$ is
\[
\laccr\left(\src,\varepsilon,p\right)\isdef\sup_{\noims\in B\left(\exams,\varepsilon\right)}\left|\left[\jac_{\ivm}(\exams)\left(\noims-\exams\right)\right]_{p}\right|
\]
where $\jac_{\ivm}\left(\exams\right)$ is the Jacobian of $\ivm$
at the point $\exams$ and $\left[\vec{v}\right]_{p}$ is the $p$-th
component of the vector $\vec{v}$.
\end{definition}
\begin{minipage}[t]{1\columnwidth}%
\end{minipage}

In \prettyref{thm:local-lipshitz-estimates} below we estimate the
function $\laccr$. The key technical tool is the following factorization
of the Jacobian of $\fwm$ which separates the nonlinear part depending
on the nodes $\{\jp_{j}\}$ from the linear part which depends on
the magnitudes $\{\jc_{i,j}\}$.
\begin{lemma}
\label{lem:jacobian-factorization} Let $\vec{x}=\left(\{\jc_{ij}\},\{\jp_{i}\}\right)\in\complexfield^{\nparams}$.
Then 
\begin{equation}
\jac_{\fwm}(\vec{x})=\cvand(\jp_{1},l_{1}+1,\dotsc,\jp_{\np},l_{\np}+1)\cdot\diag\{D_{1},\dotsc,D_{\np}\}
\end{equation}
 where $\cvand(\dotsc)$ is the confluent Vandermonde matrix \eqref{eq:confluent-vandermonde-def},
and $D_{i}$ is the $(l_{i}+1)\times(l_{i}+1)$ block 
\begin{equation}
D_{i}\isdef\begin{bmatrix}1 & 0 & 0 & \cdots & 0\\
0 & 1 & 0 & \cdots & \jc_{i,0}\\
\vdots & \vdots & \vdots & \ \ddots & \vdots\\
0 & 0 & 0 & \cdots & \jc_{i,l_{i}-1}
\end{bmatrix}.\label{eq:block-mat-def}
\end{equation}
\end{lemma}
\begin{proof}
We have by \eqref{eq:prony-map-def} 
\begin{align*}
\begin{split}\frac{\partial m_{k}}{\partial\jc_{ij}} & =\ff{k}{j}\jp_{i}^{k-j},\\
\frac{\partial m_{k}}{\partial\jp_{i}} & =\sum_{j=0}^{l_{i}-1}\jc_{ij}\ff{k}{j}(k-j)\jp_{i}^{k-(j+1)}=\sum_{j=1}^{l_{i}}\jc_{i,j-1}\ff{k}{j}\jp_{i}^{k-j}.
\end{split}
\end{align*}
 The rest of the proof is just a straightforward calculation. \end{proof}
\begin{corollary}
\label{cor:critical-points} $\vec{x}=\left(\{\jc_{ij}\},\{\jp_{i}\}\right)\in\complexfield^{\nparams}$
is a critical point of $\fwm$ if and only if at least one of the
following conditions is satisfied:
\begin{enumerate}
\item $\jp_{i}=\jp_{j}$ for any pair of indices $i\neq j$.
\item $\jc_{i,l_{i}-1}=0$ for any $1\leq i\leq\np$.
\end{enumerate}
\end{corollary}
\begin{minipage}[t]{1\columnwidth}%
\end{minipage}
\begin{corollary}
\label{cor:inverse-differential} Let $\src\in\complexfield^{\nparams}$
be a regular point of $\fwm$. Then the Jacobian matrix of the inverse
function $\ivm=\fwm^{-1}$ at $\exams=\fwm(\src)$ is equal to 
\begin{align*}
\jac_{\ivm}(\exams)=\left\{ \jac_{\fwm}\left(\src\right)\right\} ^{-1} & =\frac{\partial(\jc_{10},\dotsc,\jc_{1,l_{1}-1},\jp_{1},\dotsc,\jc_{\np,0},\dotsc,\jc_{\np,l_{\np}-1},\jp_{\np})}{\partial(m_{0},\dotsc,m_{\nparams-1})}\\
 & =\diag\{D_{1}^{-1},\dotsc,D_{\np}^{-1}\}\cdot U^{-1}(\jp_{1},l_{1}+1,\dotsc,\jp_{\np},l_{\np}+1)
\end{align*}
 where 
\begin{equation}
D_{i}^{-1}=\begin{bmatrix}1 & 0 & 0 & \cdots & 0\\
0 & 1 & 0 & \cdots & (-1)^{l_{i}-1}\frac{\jc_{i,0}}{\jc_{i,l_{i}-1}}\\
\vdots & \vdots & \vdots & \ddots & \vdots\\
0 & 0 & 0 & \cdots & \frac{1}{\jc_{i,l_{i}-1}}
\end{bmatrix}.\label{eq:inv-blocks}
\end{equation}
\begin{minipage}[t]{1\columnwidth}%
\end{minipage}
\end{corollary}
 \newcommandx\derc[1][usedefault, addprefix=\global, 1=k]{\ensuremath{\frac{\partial\jc_{ij}}{\partial m_{#1}}}}
\newcommandx\derj[1][usedefault, addprefix=\global, 1=k]{\ensuremath{\frac{\partial\jp_{i}}{\partial m_{#1}}}}

Now we are ready to formulate and prove our local stability result.
\begin{theorem}
\label{thm:local-lipshitz-estimates} Assume $\nmeas=\nparams$. Let
$\src=\left(\{\jc_{ij}\},\{\jp_{i}\}\right)\in\complexfield^{n}$
be a regular point of $\fwm$ and assume $\varepsilon$ to be small
enough so that that the inverse function $\ivm=\fwm^{-1}$ exists
in $\varepsilon$-neighborhood of $\exams=\fwm\left(\src\right)$.

Then there exists a positive constant $C_{1}$ depending only on $\jp_{1},\dotsc,\jp_{\np}$
and  $l_{1},\dotsc,l_{\np}$ such that for all $i=1,\dots,\np$ 
\begin{align*}
\laccr\left(\src,\varepsilon,\jc_{ij}\right) & =\begin{cases}
C_{1}\varepsilon & j=0\\
C_{1}\varepsilon\biggl(1+\frac{|\jc_{i,j-1}|}{|\jc_{i,l_{i}-1}|}\biggr) & 1\leq j\leq l_{i}-1
\end{cases},\\
\laccr\left(\src,\varepsilon,\jp_{i}\right) & =C_{1}\varepsilon\frac{1}{|\jc_{i,l_{i}-1}|}.
\end{align*}
\end{theorem}
\begin{proof}
Express the Jacobian matrix $\jac_{\ivm}(\exams)$ as \global\long\def\coefrow{\ensuremath{\vec{s}}}
 \global\long\def\jumprow{\ensuremath{\vec{t}}}
 
\[
\jac_{\ivm}(\exams)=\begin{bmatrix}\coefrow_{10}^{T} & \dotsc & \coefrow_{1,l_{1}-1}^{T} & \jumprow_{1}^{T} & \dotsc & \coefrow_{n0}^{T} & \dotsc & \coefrow_{\np,l_{\np}-1}^{T} & \jumprow_{\np}^{T}\end{bmatrix}^{T}
\]
 where 
\begin{align*}
\coefrow_{ij} & \isdef\begin{bmatrix}\derc[0] & \derc[1] & \dotsc & \derc[\nmeas-1]\end{bmatrix},\\
\jumprow_{i} & \isdef\begin{bmatrix}\derj[0] & \derj[1] & \dotsc\derj[\nmeas-1]\end{bmatrix}.
\end{align*}
Let $\noims=\left(\meas[0]+\dl\meas[0],\dots,\meas[\nmeas-1]+\dl\meas[\nmeas-1]\right)$
where each $\left|\dl\meas\right|<\varepsilon$. Denote by $\|\cdot\|_{1}$
the $l_{1}$ vector norm, i.e. if $\vec{v}=(v_{i})$ is an $n$-vector
then $\|\vec{v}\|_{1}\isdef\sum_{i=1}^{n}|v_{i}|$. Then 
\begin{align*}
\left[\jac_{\ivm}\left(\exams\right)\left(\noims-\exams\right)\right]_{\jc_{ij}} & =\biggl|\sum_{k=0}^{P-1}\derc\dl m_{k}\biggr|\leq\varepsilon\|\coefrow_{ij}\|_{1},\\
\left[\jac_{\ivm}\left(\exams\right)\left(\noims-\exams\right)\right]_{\jp_{i}} & =\biggl|\sum_{k=0}^{P-1}\derj\dl m_{k}\biggr|\leq\varepsilon\|\jumprow_{i}\|_{1}.
\end{align*}
By Corollary \ref{cor:inverse-differential}, the matrix $\jac_{\ivm}$
is the product of the block diagonal matrix $D^{*}\isdef\diag\{D_{1}^{-1},\dotsc,D_{\np}^{-1}\}$
with the matrix $\cvand^{*}\isdef(\cvand(\jp_{1},l_{1}+1,\dotsc,\jp_{\np},l_{\np}+1))^{-1}$.
Therefore, $\coefrow_{ij}$ and $\jumprow_{i}$ are the products of
the corresponding rows of $D_{i}^{-1}$ with $U^{*}$. Let $D_{i}^{-1}=(d_{k,l}^{(i)})$
and $\cvand^{*}=(u_{k,l})$. Then: 
\begin{align*}
\|\coefrow_{ij}\|_{1}=\sum_{k=1}^{P}\biggl|\sum_{l=1}^{l_{i}+1}d_{j,l}^{(i)}u_{l,k}\biggr|\leq\sum_{l=1}^{l_{i}+1}|d_{j,l}^{(i)}|\sum_{k=1}^{P}|u_{l,k}|
\end{align*}
 and likewise 
\begin{align*}
\|\jumprow_{i}\|_{1}\leq\sum_{l=1}^{l_{i}+1}|d_{l_{i}+1,l}^{(i)}|\sum_{k=1}^{P}|u_{l,k}|.
\end{align*}
Let $\|\cdot\|_{\infty}$ denote the ``maximal row sum'' matrix
norm -- i.e. for any $n\times n$ matrix $C=(c_{ij})$ we have $\|C\|_{\infty}\isdef\max_{i=1,\dotsc,n}\sum_{j=1}^{n}|c_{ij}|$.

Denote $C_{1}\isdef\|\cvand^{*}\|_{\infty}$. Then substitute for
$d_{l,k}^{(i)}$ the actual entries of $D_{i}^{-1}$ from \eqref{eq:inv-blocks}
into the above and get the desired result.
\end{proof}

\section{\label{sec:known-results}Comparison with known results}

\subsection{\label{sub:crb-pace}CRB for PACE model}

The confluent Prony system \eqref{eq:gen-prony} is equivalent to
the PACE model \cite{badeau2006high,badeau2008cramer}. The Cramer-Rao
bound (CRB) (which gives a lower bound for the variance of any unbiased
estimator) of the PACE model in colored Gaussian noise is as follows
(note that the original expressions have been appropriately modified
to match the notations of this paper).
\begin{theorem}[{\cite[Proposition III.1]{badeau2008cramer}}]
\label{thm:crb}Let the noise variance be $\sigma^{2}$, then%
\footnote{Here $\mathfrak{R}\left(\cdot\right)$ denotes the real part.%
}
\begin{align*}
CRB\left\{ \jp_{i}\right\}  & =C_{2}\frac{\sigma^{2}}{\left|\jp_{i}\right|^{2}\left|\jc_{i,l_{i}-1}\right|^{2}},\\
CRB\left\{ \jc_{i,0}\right\}  & =C_{3}\sigma^{2},\\
CRB\left\{ \jc_{i,j}\right\}  & =C_{4}\sigma^{2}\left(C_{5}\left|\frac{\jc_{i,j-1}}{\jc_{i,l_{i}-1}}\right|^{2}+C_{6}\mathfrak{R}\left\{ \frac{\jc_{i,j-1}}{\jc_{i,l_{i}-1}}\right\} +1\right)\qquad j=1,2,\dots,l_{i}-1,
\end{align*}
where $C_{2},\dots,C_{6}$ are constants depending on the configuration
of the nodes $\left\{ \jp_{i}\right\} $, while in addition $C_{4},C_{5},C_{6}$
depend on the index $j$.
\end{theorem}
\begin{minipage}[t]{1\columnwidth}%
\end{minipage}

As mentioned in \prettyref{sub:related-work}, there exist several
essential differences between our setting and the statistical signal
estimation framework, in particular:
\begin{enumerate}
\item no a-priori statistical model of the noise is available;
\item no assumptions on the reconstruction algorithm (estimator) such as
unbiasedness are made;
\item measure of performance is the worst-case error rather than estimator
variance.
\end{enumerate}
The expressions for the CRB in \prettyref{thm:crb} are very similar
to the local point-wise accuracy bounds of \prettyref{thm:local-lipshitz-estimates}.
The reason for such similarity is not a-priori clear (although it
could be partially attributed to the fact that both methods require
calculation of the partial derivatives of the measurements with respect
to the parameters), and it certainly prompts for further investigation.

\subsection{\label{sub:ESPRIT}ESPRIT method}

The ESPRIT algorithm is one of the best performing subspace methods
for estimating parameters of the Prony systems with white Gaussian
noise. Originally developed in the context of frequency estimation
\cite[Section 4.7]{stoica2005spectral}, it has been generalized to
the full PACE model \cite{badeau2006high}, and its performance has
been shown to approach the CRB in the case of high SNR and infinite
observation length.

In essence, the ESPRIT (and other subspace methods) relies on the
following observations:
\begin{enumerate}
\item The range (column space) of both the data matrix $M_{\ncoeffs}$ \eqref{eq:poly-meas-hankel-system}
and the confluent Vandermonde matrix $\cvand$ \eqref{eq:confluent-vandermonde-def}
are the same (follows directly from \eqref{eq:prony-hankel-main-factorization});
\item the matrix $\cvand$ has the so-called \emph{rotational invariance
property} (\cite{badeau2006high}):
\[
\cvand^{\uparrow}=\cvand_{\downarrow}J
\]
where $\cvand^{\uparrow}$ denotes $\cvand$ without the first row,
$\cvand_{\downarrow}$ denotes $\cvand$ without the last row, and
$J$ is a block diagonal matrix whose $i$-th block is the $l_{i}\times l_{i}$
Jordan block with the number $\jp_{i}$ on the diagonal.
\end{enumerate}
Suppose we knew $\cvand$, then the matrix $J$ could be found by
\[
J=\cvand_{\downarrow}{}^{\sharp}\cvand^{\uparrow}
\]
(where $^{\#}$ denotes the Moore-Penrose pseudo-inverse) and then
the nodes $\jp_{j}$ could be recovered as the eigenvalues of $J$.

Unfortunately, $\cvand$ is unknown in advance, but suppose we had
at our disposal a matrix $W$ whose column space was identical to
that of $\cvand$. In that case, we would have $W=\cvand G$ for an
invertible $G$, and consequently
\[
W^{\uparrow}=W_{\downarrow}\Phi
\]
where
\[
\Phi=G^{-1}JG
\]
which means that the eigenvalues of $\Phi$ are also $\left\{ \jp_{i}\right\} $.
Such a matrix $W$ can be obtained for example from the singular value
decomposition (SVD) of the data matrix/covariance matrix. To summarize,
the ESPRIT method for estimating $\left\{ \jp_{i}\right\} $, as used
in our experiments below, is as follows.

\begin{algorithm}[H]
Let $M_{\nmeas}$ be a rectangular $n\times l$ Hankel matrix built
from the measurements.
\begin{enumerate}
\item Compute the SVD $M_{\nmeas}=W\Sigma V^{T}$.
\item Calculate $\Phi=W_{\downarrow}^{\#}W^{\uparrow}$.
\item Set $\left\{ \jp_{i}\right\} $ to be the eigenvalues of $\Phi$ with
appropriate multiplicities (use e.g. arithmetic means to estimate
multiple nodes which are scattered by the noise).
\end{enumerate}
\caption{ESPRIT method for recovering the nodes $\left\{ \jp_{i}\right\} $.}

\label{alg:esprit}
\end{algorithm}

Note that the dimensions $n,l$ are not fixed a-priori, but in \cite{badeau2008performance}
it is shown that taking $n=2l$ or $l=2n$ results in optimal performance
for non-confluent Prony system \eqref{eq:basic-prony}.

Since the performance of the ESPRIT method is close to the CRB which,
in turn, resembles our local bounds, we regard the ESPRIT as the best
candidate among the ``global'' solution methods of the confluent
Prony system. It should be noted, however, that the analysis of ESPRIT
as presented in \cite{badeau2008performance} suggests a relatively
complicated dependence of the estimator performance on the model parameters
for small number of measurements $\nmeas$.

\subsection{\label{sub:apm}Approximate Prony method}

In \cite{potts2010parameter} the authors develop the Approximate
Prony method for solving the system \eqref{eq:basic-prony} (restricting
$\jp_{j}$ to be of unit length), and analyze its performance for
small measurement errors. In more detail, the model is defined as
\[
h\left(x\right)=\sum_{j=1}^{M}c_{j}\ee^{\imath f_{j}x}\qquad x\in\reals,\; c_{j}\in\complexfield,\; f_{j}\in\left(-\pi,\pi\right).
\]
The measurements are given with errors
\[
\widetilde{h}\left(k\right)=h\left(k\right)+e_{k},\qquad k=0,\dots,2N
\]
 where the number of measurements $N$ satisfies $N\geq2M+1$. Finally,
the coefficients $c_{j}$ are assumed to be large with respect to
the noise level, i.e.
\[
\left|e_{k}\right|\leq\varepsilon_{1}\ll\left|c_{j}\right|.
\]

The proposed solution method is as follows.

\begin{algorithm}[H]
\begin{enumerate}
\item \label{step:apm-frequency-estimate}Build the Hankel matrix $\widetilde{H}\in\complexfield^{2N-L,L}$
from the measurements where $L$ is an upper bound on the number of
nodes. Compute singular value decomposition of $\widetilde{H}$, and
take the smallest nonzero singular value and its singular vector $\vec{v}=\left(v_{i}\right).$
Finally, compute the roots of the polynomial $p\left(z\right)=\sum_{i=0}^{L}v_{i}z^{i}$.
These are the approximations of $\left\{ f_{j}\right\} .$
\item \label{step:amplitudes-recovery}Find $\left\{ c_{j}\right\} $ by
solving an overdetermined Vandermonde linear system.
\end{enumerate}
\caption{Approximate Prony method.}

\label{alg:apm}
\end{algorithm}

The stability analysis of the APM is performed only for the \prettyref{step:amplitudes-recovery}
above, assuming that the frequencies $\left\{ f_{j}\right\} $ have
been recovered with high accuracy. \cite[Theorem 5.2]{potts2010parameter}
gives the following estimate:
\begin{equation}
\left|c_{j}-\widetilde{c}_{j}\right|\sim\sqrt{NM}\left|f_{j}-\widetilde{f}_{j}\right|\max_{k}\left|h_{k}\right|+\max_{k}\left|\Delta h_{k}\right|.\label{eq:apm-accuracy}
\end{equation}

While missing explicit analysis of \prettyref{step:apm-frequency-estimate}
above (however, the actual numerical accuracy of this step was shown
in \cite{peter2011nonlinear} to be comparable to the performance
of the ESPRIT method) and dealing with single poles only, these results
may provide an important insight as to the dependence of the accuracy
on the number of measurements $N$, as well as to the applicability
of the Vandermonde inversion for recovering the magnitudes (the errors
in fact \emph{increase} with $N$!) In addition, the authors notice
that the accurate recovery of the magnitudes depends greatly on a
sufficient accuracy of recovering the nodes, and this fact is also
reflected in our numerical experiments (\prettyref{sec:experiments}).

\section{\label{sec:experiments}Numerical experiments}

In our numerical experiments we had two distinct goals:
\begin{enumerate}
\item Numerically investigate the ``best possible local accuracy'' of inverting
\eqref{eq:gen-prony} as a function of the various parameters of the
problem, and compare the results with the predictions of \prettyref{thm:local-lipshitz-estimates}.
\item Ascertain whether there exist some regular patterns in the behavior
of the global solution methods (Prony and ESPRIT) in a similar ``local''
setting, and compare their performance to the optimal one.
\end{enumerate}

\subsection{Experimental setup}
\begin{enumerate}
\item Given $\np,\ord$, choose the jumps $\jp_{1},\dots,\jp_{\np}\in\left[0,1\right]$
and the magnitudes $\jc_{1,0},\dots,\jc_{\np,\ord-1}\in\left[-1,1\right]$.
\item Change one or more of the parameters according to a particular experiment.
\item Calculate the perturbed moments $\nn m_{k}=m_{k}+\varepsilon_{k}$
where $m_{k}$ is given by \eqref{eq:gen-prony} and $\varepsilon_{k}\ll1$
(on the order of $10^{-10})$ are randomly chosen. 
\item Invert \eqref{eq:gen-prony} with the right hand side given by $\nn m_{k}$
by one of the three methods:

\begin{enumerate}
\item Nonlinear least squares minimization (using MATLAB's \code{lsqnonlin}
routine) with the initial guess being very close to the true parameter
values. This is our simulation of the ``local'' setting.
\item Global Prony method - \prettyref{alg:prony-method}.
\item ESPRIT method - \prettyref{alg:esprit}.
\end{enumerate}
\item Calculate the absolute errors $\left|\Delta\jp_{j}\right|=\left|\jp_{j}-\nn\jp_{j}\right|$
and $\left|\Delta\jc_{i,j}\right|=\left|\jc_{i,j}-\nn\jc_{i,j}\right|$.
\end{enumerate}
In all the experiments we took $\np=2$. All solution methods were
applied \emph{to the same moment sequence $\left\{ \meas\right\} $.
}The number of measurements is the minimal necessary for exact inversion,
namely $\nparams$ for least squares and $2\ncoeffs$ both for Prony
and ESPRIT.

\subsection{\label{sub:numerical-results}Results}

\begin{figure}
\subfloat[Least squares]{\includegraphics[bb=66bp 191bp 549bp 616bp,clip,scale=0.45]{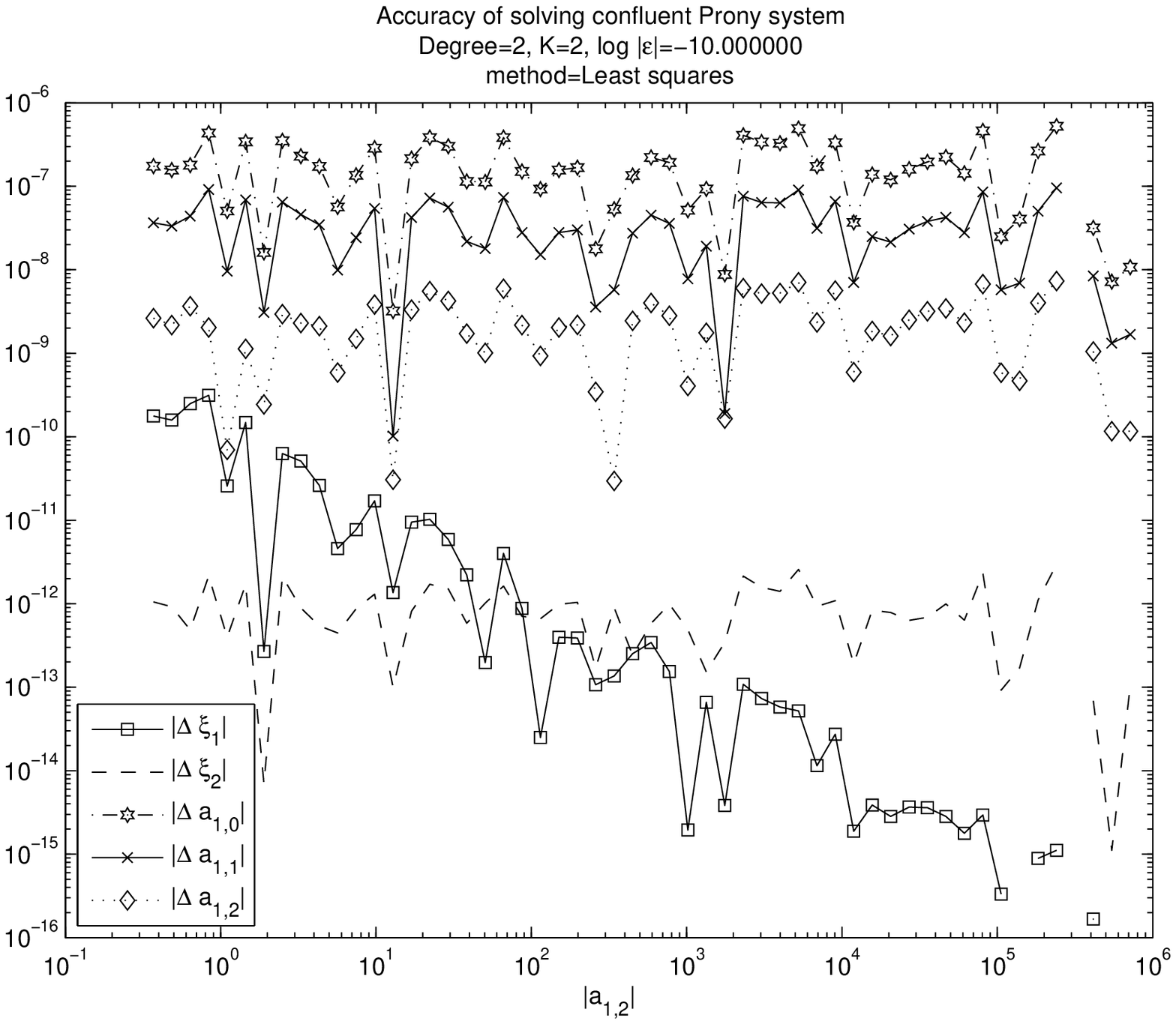}\label{subfig:highest-deg2-ls}}
\subfloat[Prony]{\includegraphics[bb=66bp 191bp 549bp 616bp,clip,scale=0.45]{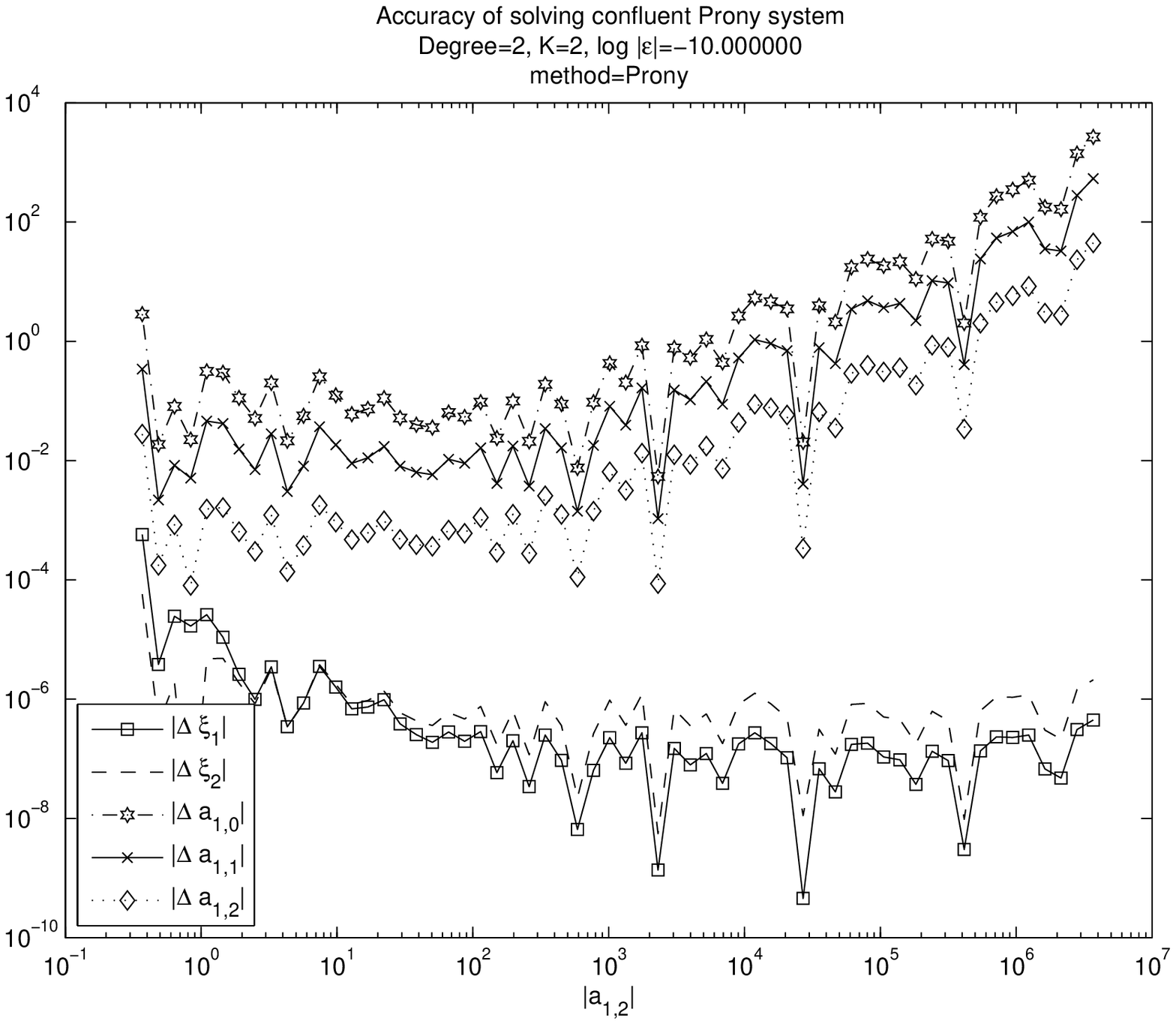}\label{subfig:highest-coeff-prony}} 

\subfloat[ESPRIT]{\includegraphics[bb=66bp 191bp 549bp 616bp,clip,scale=0.45]{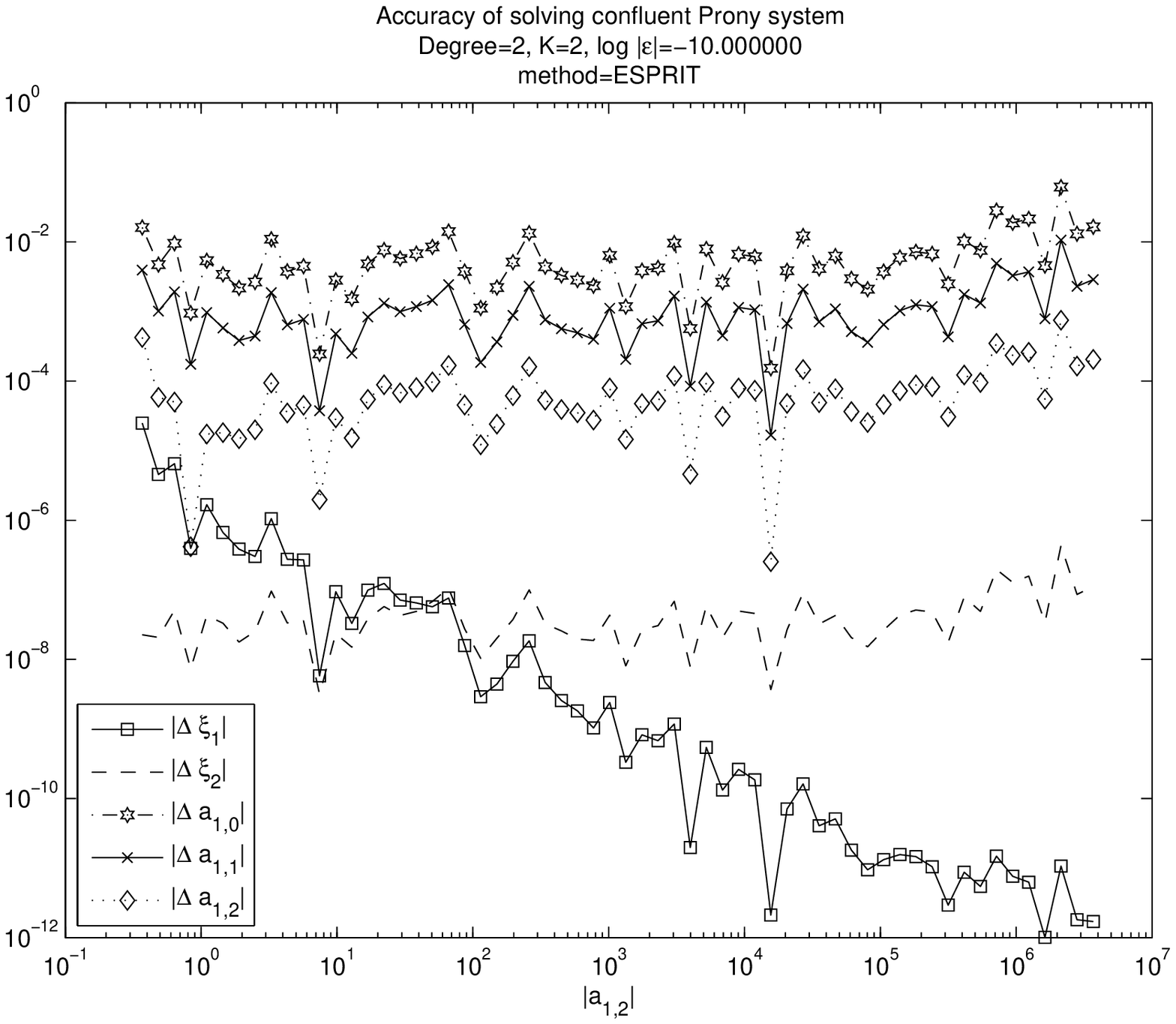}}
\subfloat[Least squares. Note the growth of $\left|\Delta\jc_{1,1}\right|$.]{\includegraphics[bb=66bp 191bp 549bp 616bp,clip,scale=0.45]{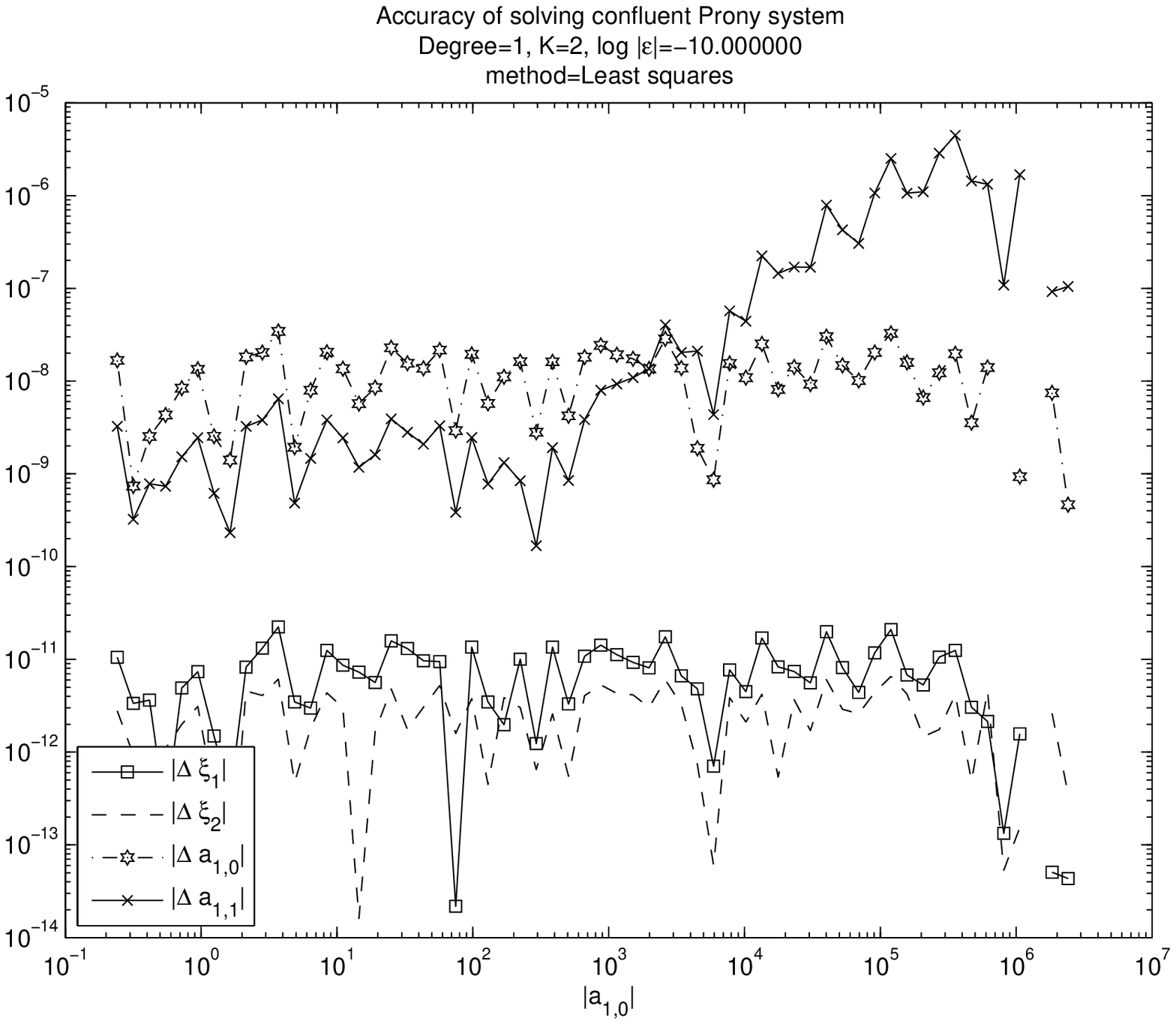}

\label{subfig:other-ls-d1}}

\subfloat[Prony]{\includegraphics[bb=66bp 191bp 549bp 616bp,clip,scale=0.45]{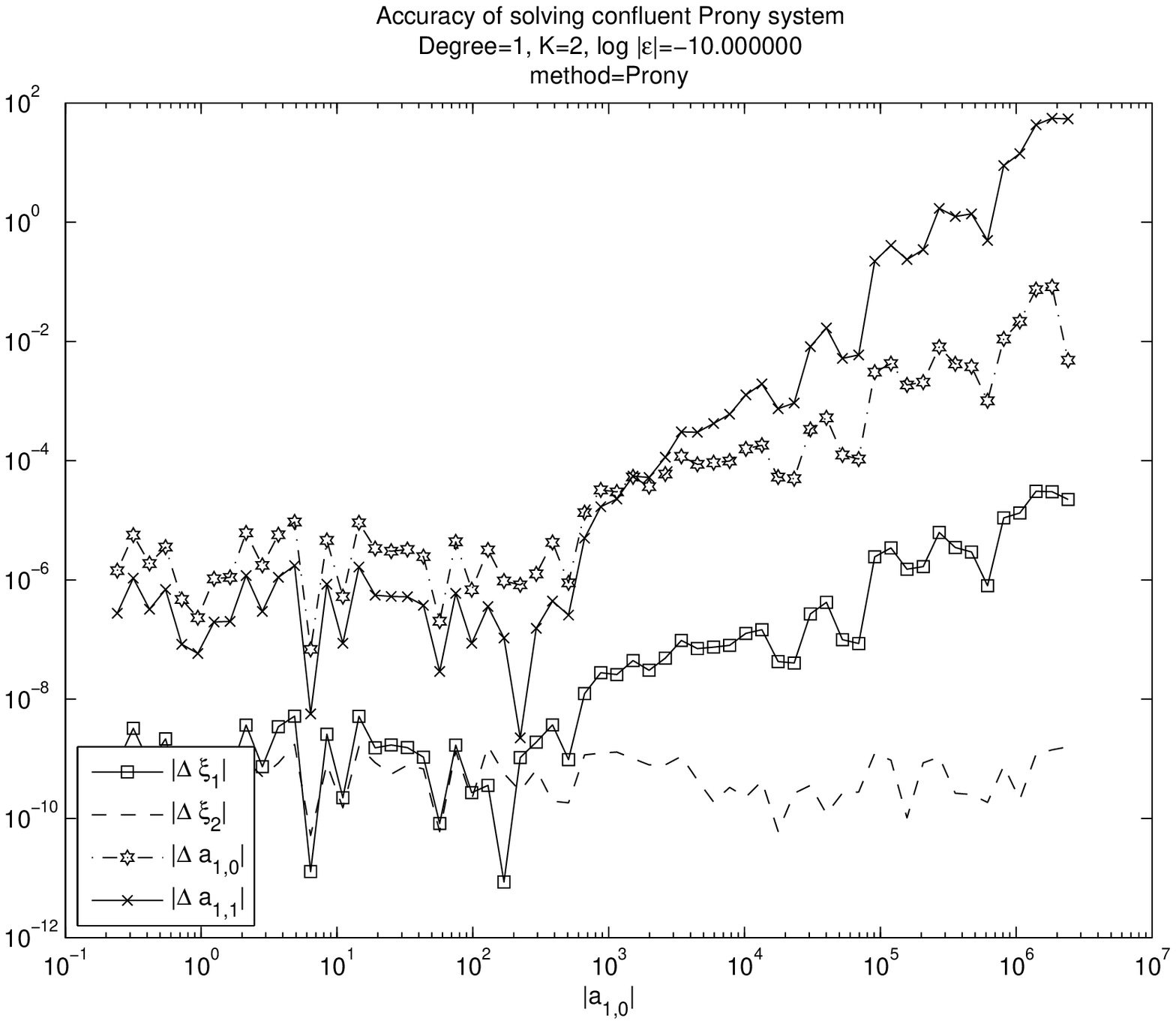}

} \subfloat[ESPRIT]{\includegraphics[bb=66bp 191bp 549bp 616bp,clip,scale=0.45]{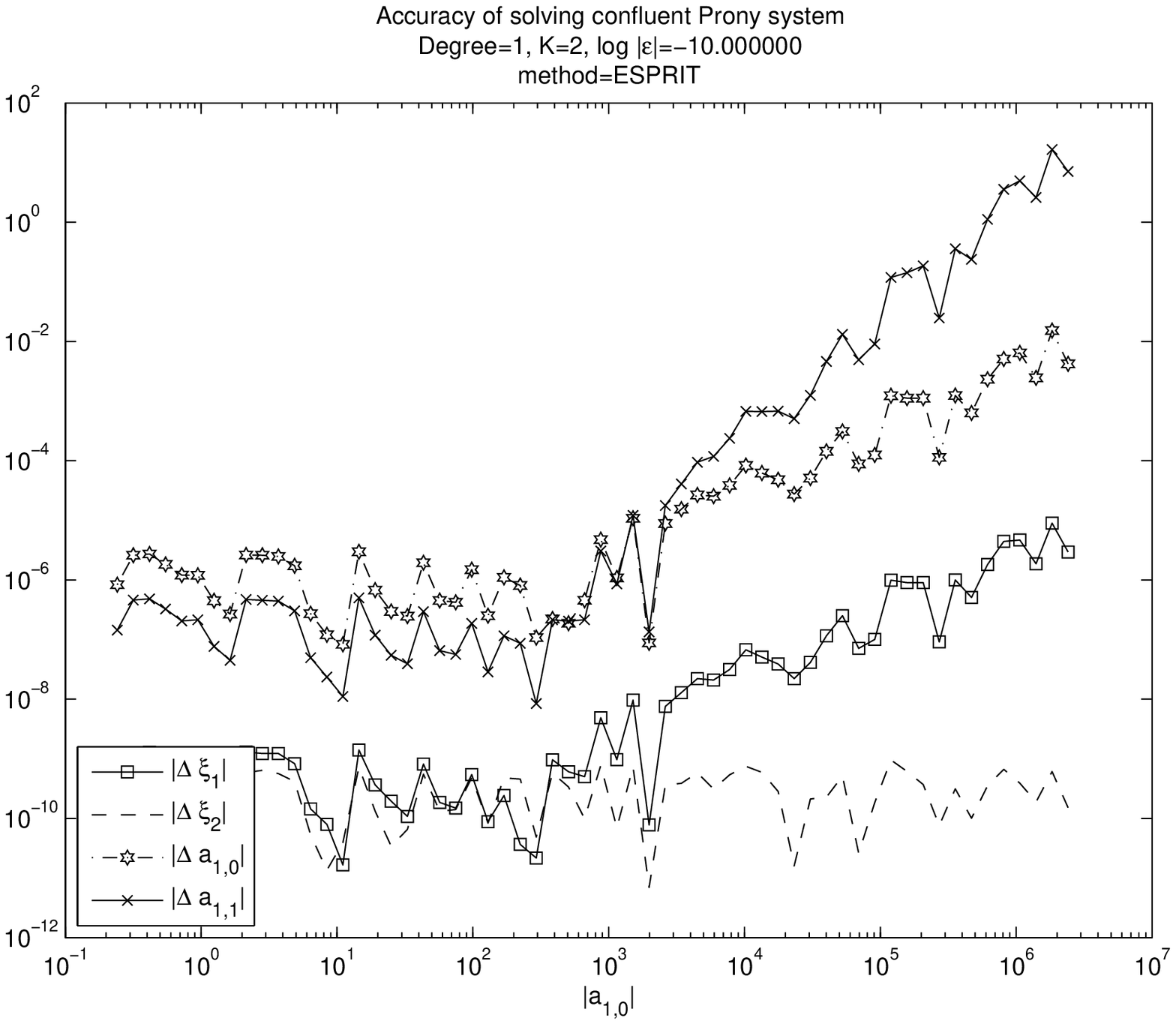}

} \caption{{\footnotesize (a-c):} {\footnotesize Dependence of the reconstruction
error on the magnitude of the highest coefficient, degree = 2.}\protect \\
{\footnotesize (d-f):} {\footnotesize Dependence of the reconstruction
error on the magnitude of the ``previous'' coefficient, degree =
1.}}

\label{fig:coeff-change}
\end{figure}

\subsubsection{Changing the highest coefficient}

In the first set of experiments, we checked how the reconstruction
errors $\left|\Delta\jp_{i}\right|,\left|\Delta\jc_{i,j}\right|$
depend on the magnitude of the highest coefficient $|\jc_{i,l_{i}-1}|$.
The results are presented in \prettyref{fig:coeff-change} (a-c).

For both least squares and ESPRIT (but not for Prony), the inverse
proportionality $\left|\Delta\jp_{i}\right|\sim\frac{1}{|\jc_{i,l_{i}-1}|}$
is seen in \prettyref{fig:coeff-change} (a), (c), matching the theoretical
predictions of \prettyref{thm:local-lipshitz-estimates}.

For LS and ESPRIT, the errors $\left|\Delta\jc_{i,j}\right|$ seem
to be unaffected by the increase in $|\jc_{i,l_{i}-1}|$. This can
be explained very well by the formula $\left|\Delta\jc_{i,j}\right|\sim1+\frac{\left|\jc_{i,j-1}\right|}{\left|\jc_{i,l_{i}-1}\right|}$
so that indeed $\left|\Delta\jc_{i,j}\right|$ should remain close
to constant as $|\jc_{i,l_{i}-1}|\to\infty$.

The Prony method's performance with respect to the recovery of the
magnitudes actually \emph{degrades} with the increase in $|\jc_{i,l_{i}-1}|$.
Although both Prony and ESPRIT use the same method for the recovery
of the magnitudes, it appears that the initial error in recovering
the nodes, which is significantly smaller in ESPRIT (see \prettyref{sub:numerical-epsilon}
below), influences this step greatly - in accordance with the predictions
of \cite{potts2010parameter,peter2011nonlinear} (see also discussion
in \prettyref{sub:apm}).

In addition, the Prony method fails to separate recovery of a node
and its magnitudes (say $\Delta\jp_{1},\Delta\jc_{1,j}$) from the
highest magnitude associated with \emph{another} node (e.g. $\left|\jc_{2,l_{2}-1}\right|$)
- these results are not shown for saving space.

\subsubsection{Changing coefficient other than the highest}

In the second set of experiments, we changed the magnitude of some
coefficient other than the highest, i.e. $\jc_{i,j}$ for $j<l_{i}-1$.
The results are presented in \prettyref{fig:coeff-change} (d-f).

For the least squares method, the dependence of $\left|\Delta\jc_{i,j}\right|$
on the ``previous'' magnitude $\left|\jc_{i,j-1}\right|$ for $j\neq0$
is consistent with the formula $\left|\Delta\jc_{i,j}\right|\sim1+\frac{\left|\jc_{i,j-1}\right|}{\left|\jc_{i,l_{i}-1}\right|}$
- such a behavior should be visible when $\left|\jc_{i,j-1}\right|\gg\left|\jc_{i,l_{i}-1}\right|$,
as can indeed be noticed in \prettyref{subfig:other-ls-d1}. In addition,
the other magnitudes and the jumps are unaffected, as predicted.

On the contrary, neither Prony nor ESPRIT succeed in confining the
influence of $\left|\jc_{i,j-1}\right|$ only to the recovery of the
next magnitude $\left|\Delta\jc_{i,j}\right|.$ In particular, $\left|\Delta\jp_{1}\right|$
increases with $\left|\jc_{1,0}\right|$ in both of them. The error
in \emph{all the magnitudes} grows with $\left|\jc_{1,0}\right|$,
as opposed to the least squares where only $\left|\Delta\jc_{1,1}\right|$
is increased.

\begin{figure}
\subfloat[Least squares]{\includegraphics[bb=66bp 195bp 551bp 616bp,clip,width=0.33\columnwidth]{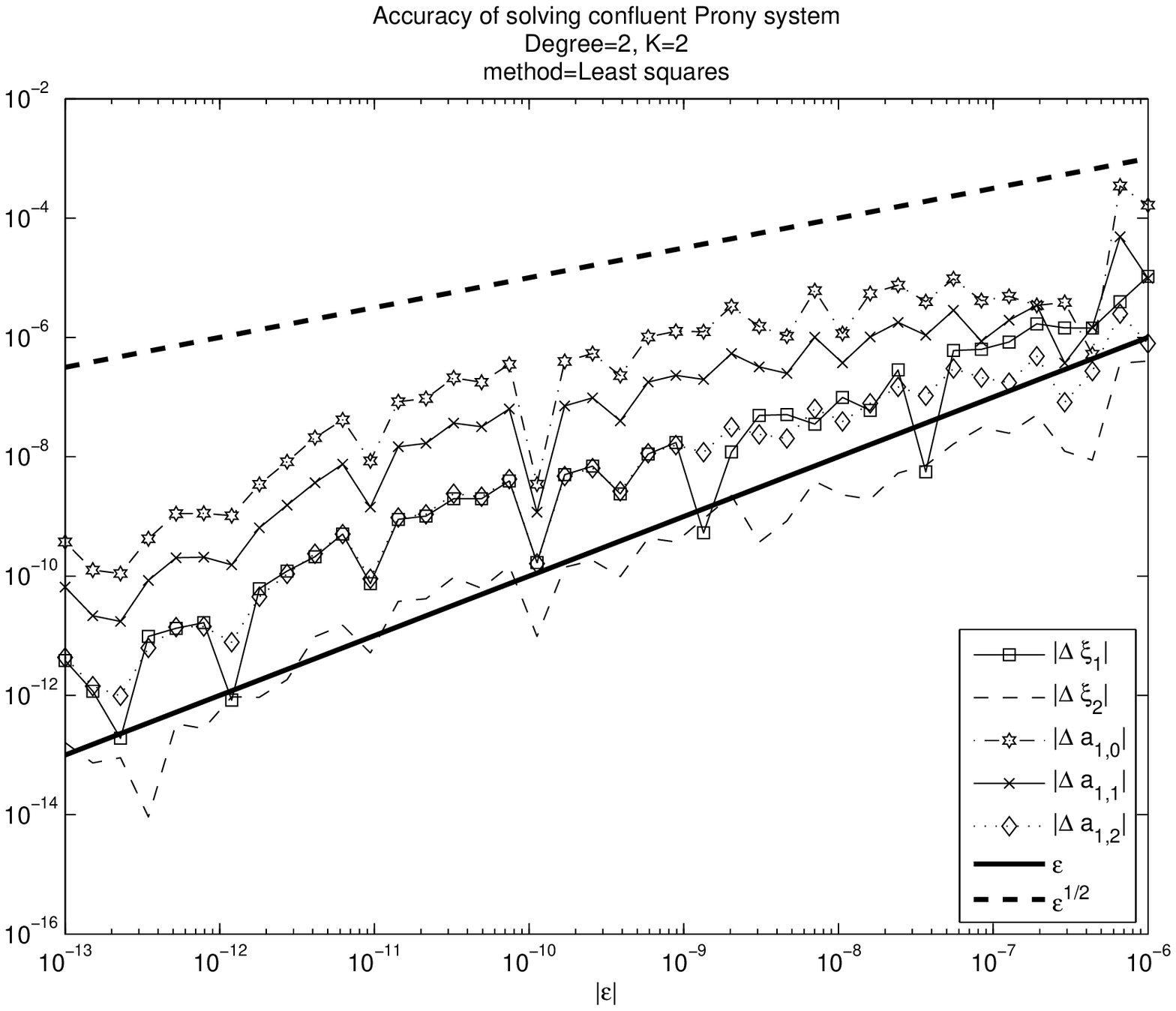}}
\subfloat[Prony]{\includegraphics[clip,width=0.33\columnwidth]{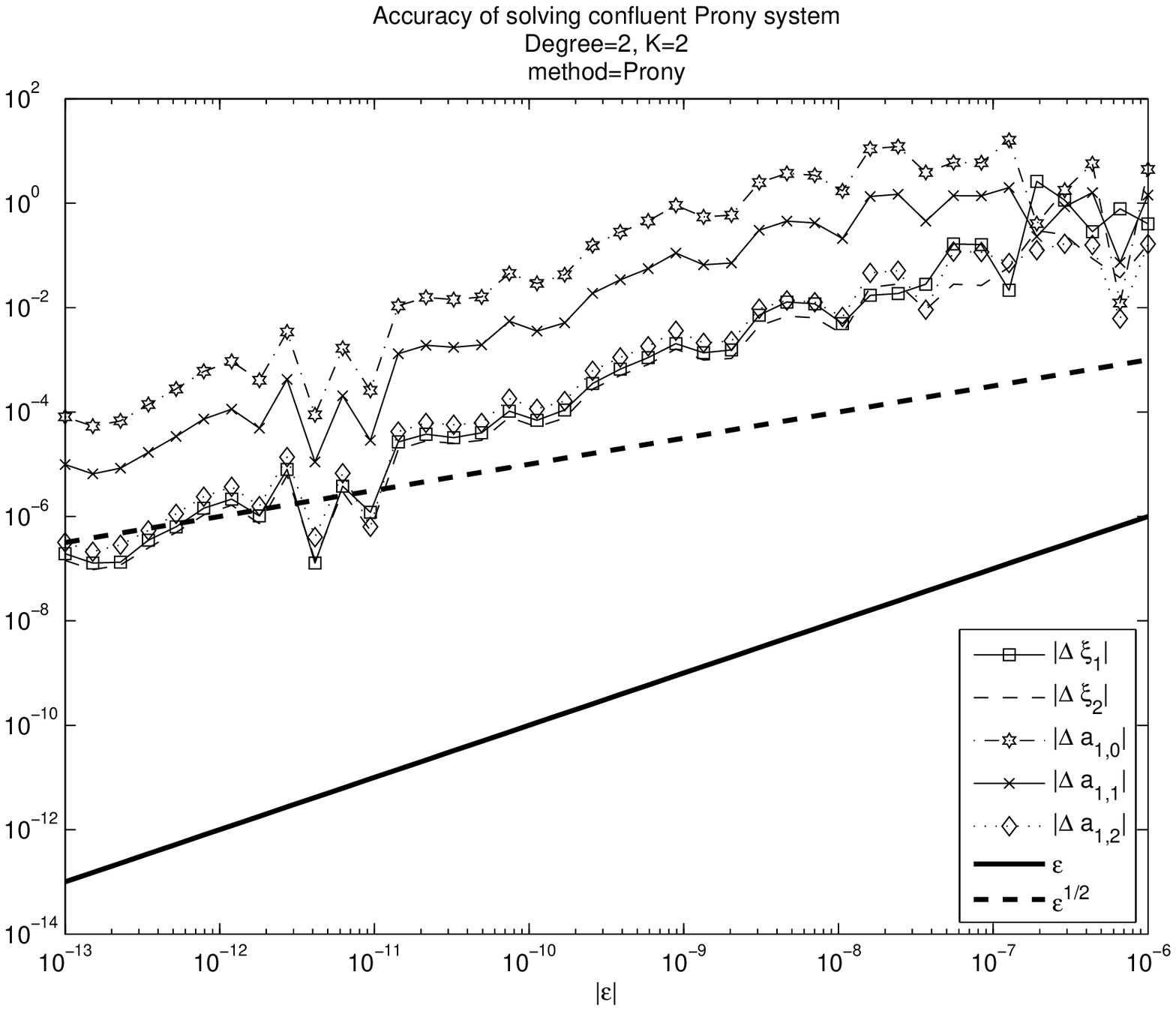}\label{subfig:eps-prony-d3}}
\subfloat[ESPRIT]{\includegraphics[clip,width=0.33\columnwidth]{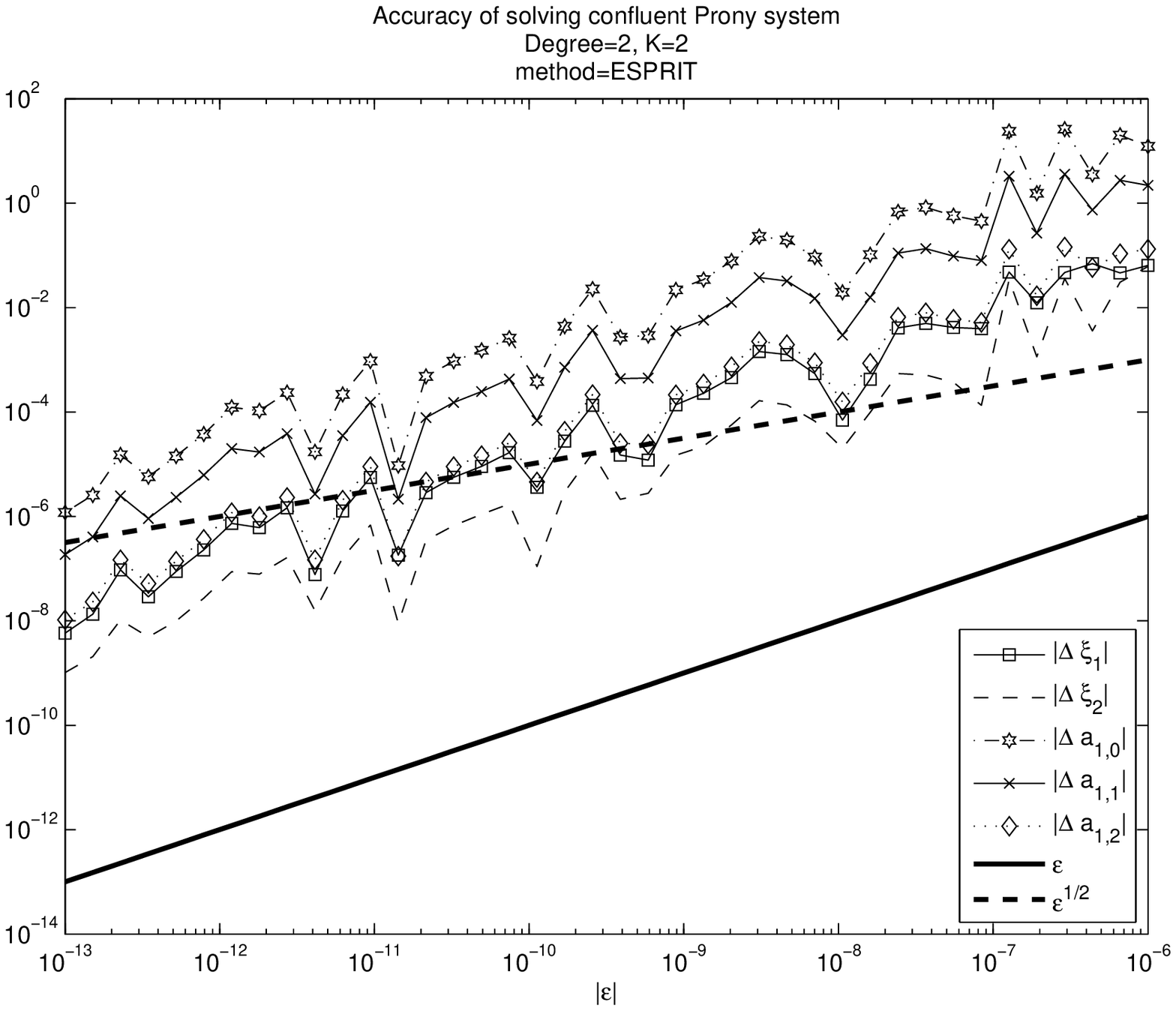}\label{subfig:eps-esprit-d3}}
\caption{Reconstruction error as $\varepsilon\to0$, degree = 2.}

\label{fig:abs-error}
\end{figure}

\subsubsection{\label{sub:numerical-epsilon}Dependence on the measurement error}

In the next experiment, we kept all the parameters constant and changed
the magnitude of the error $\max_{k}\varepsilon_{k}$. The results
are presented in \prettyref{fig:abs-error}. The ESPRIT performs slightly
better than Prony, but both of them are worse than the optimal least
squares. Note however that the asymptotic error (the slope) is $O\left(\varepsilon\right)$
in spite of the fact that both algorithms involve extraction of multiple
roots which should decrease the accuracy to $O\left(\varepsilon^{\frac{1}{d}}\right)$
where $d$ is the order of the pole. This phenomenon can be explained
by the effect of averaging the clustered roots (see \cite[Proposition V.3]{badeau2006high}).

\begin{figure}
\subfloat[Least squares]{\includegraphics[bb=66bp 195bp 551bp 616bp,clip,width=0.33\columnwidth]{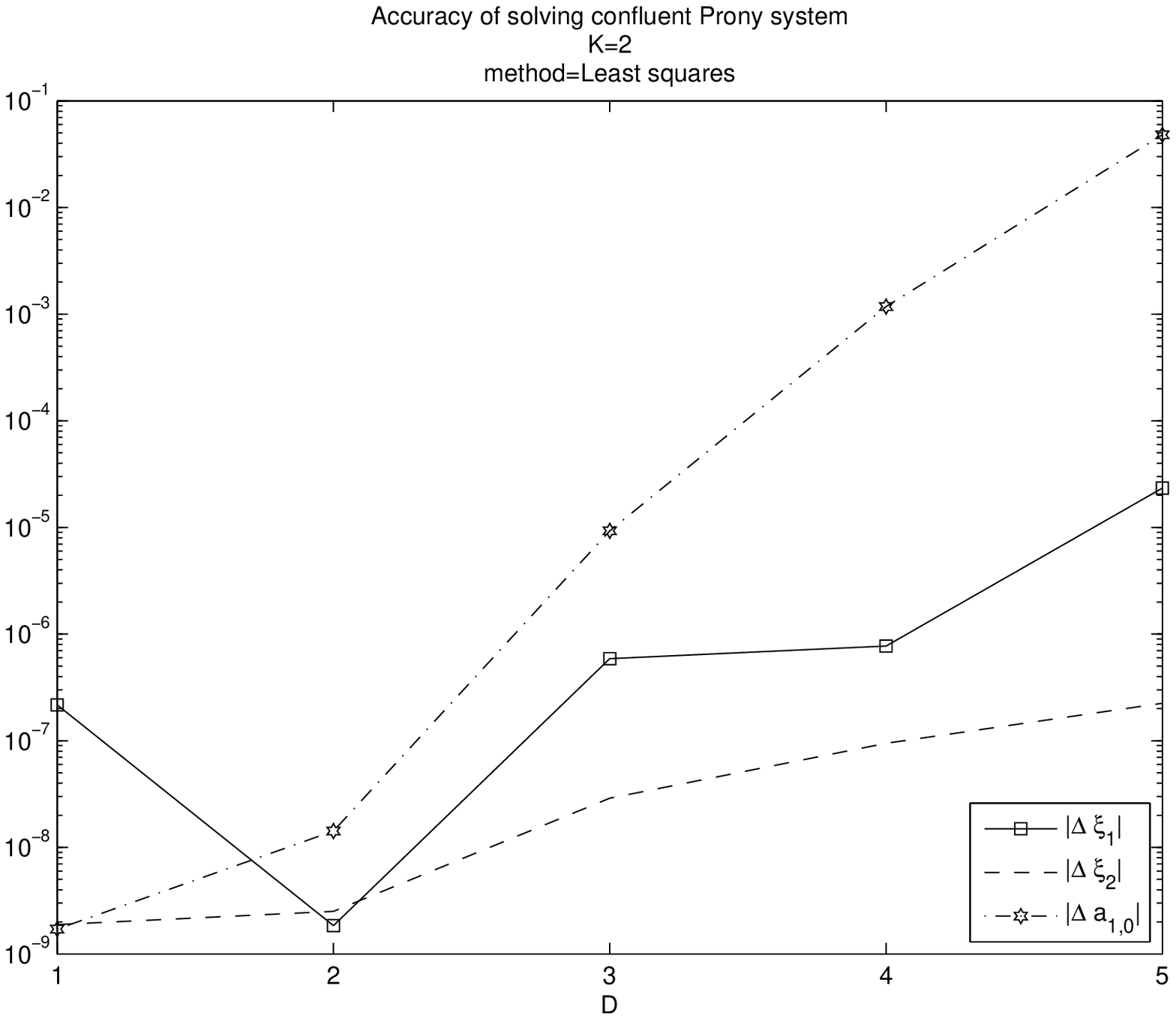}}
\subfloat[Prony]{\includegraphics[clip,width=0.33\columnwidth]{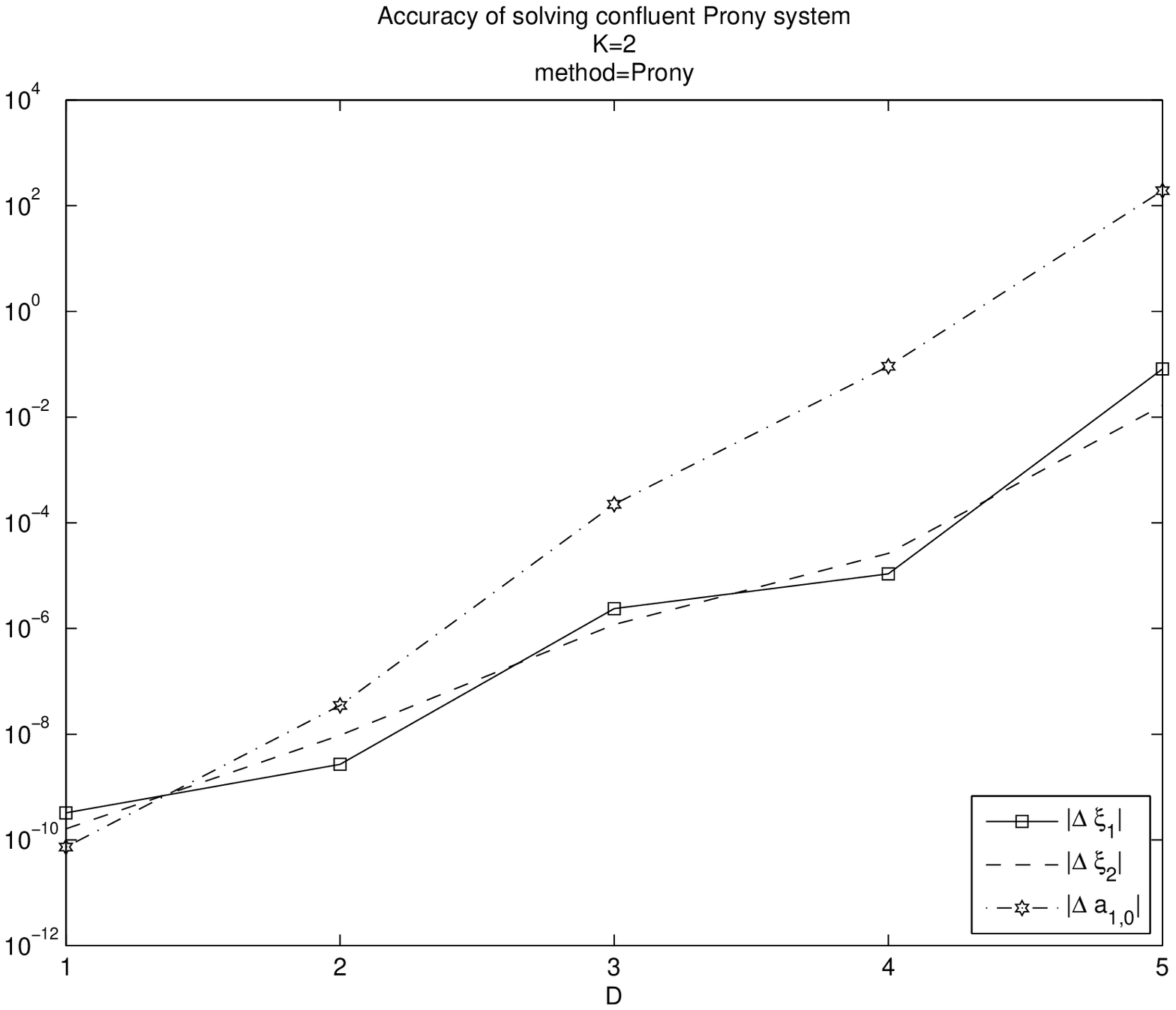}}
\subfloat[ESPRIT]{\includegraphics[clip,width=0.33\columnwidth]{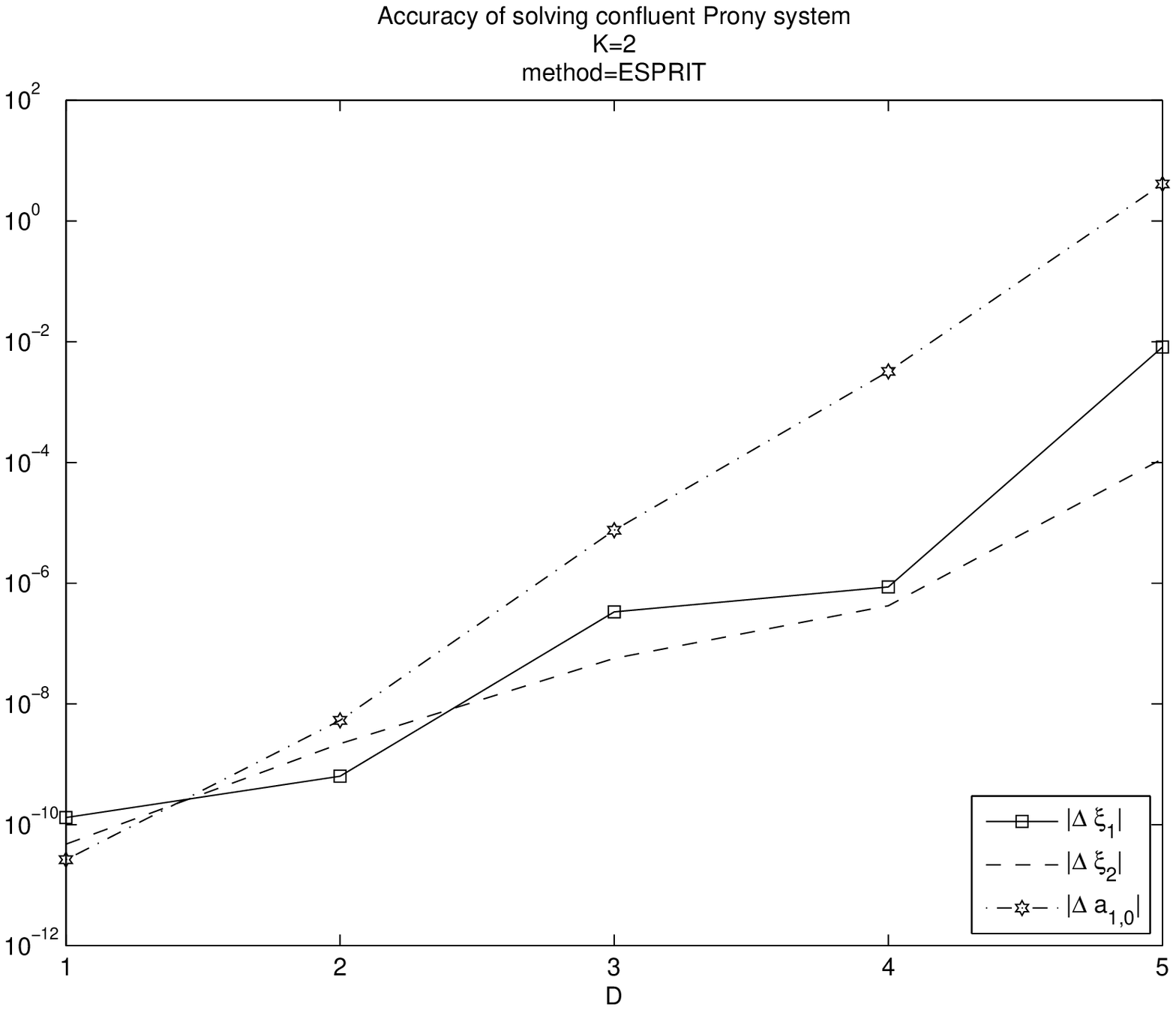}}

\caption{Dependence of the reconstruction error on the order of the model.}

\label{fig:order}
\end{figure}

\subsubsection{Dependence on the model order}

Next, we checked the dependence of the reconstruction error on the
model order $D\isdef\max_{i=1,\dots,\np}l_{i}.$ The results are presented
in \prettyref{fig:order}. The reconstruction error for all the parameters
grows exponentially in $D$ for all the methods.

\begin{figure}
\subfloat[Least squares]{\includegraphics[width=0.33\columnwidth]{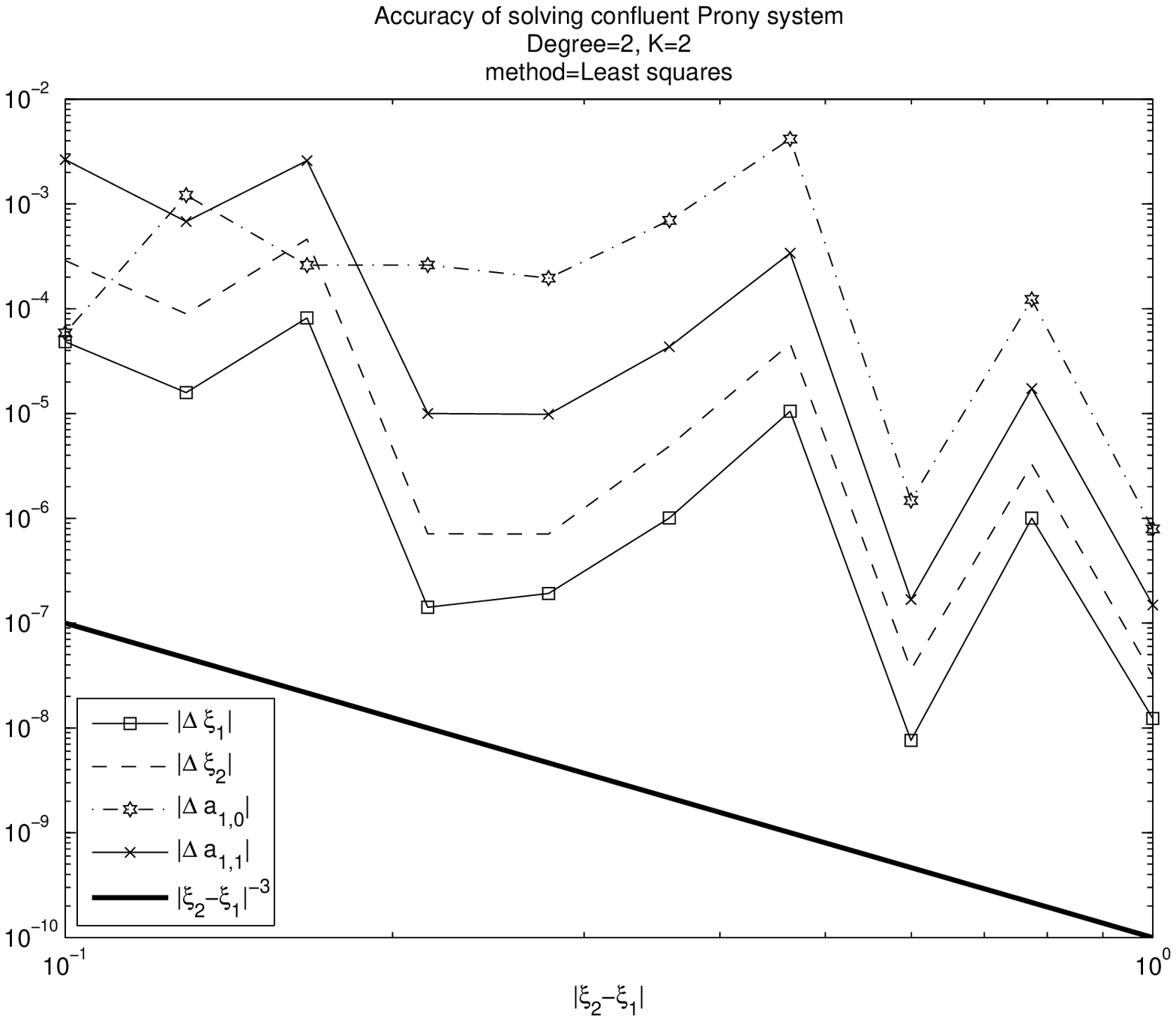}

} \subfloat[Prony]{\includegraphics[width=0.33\columnwidth]{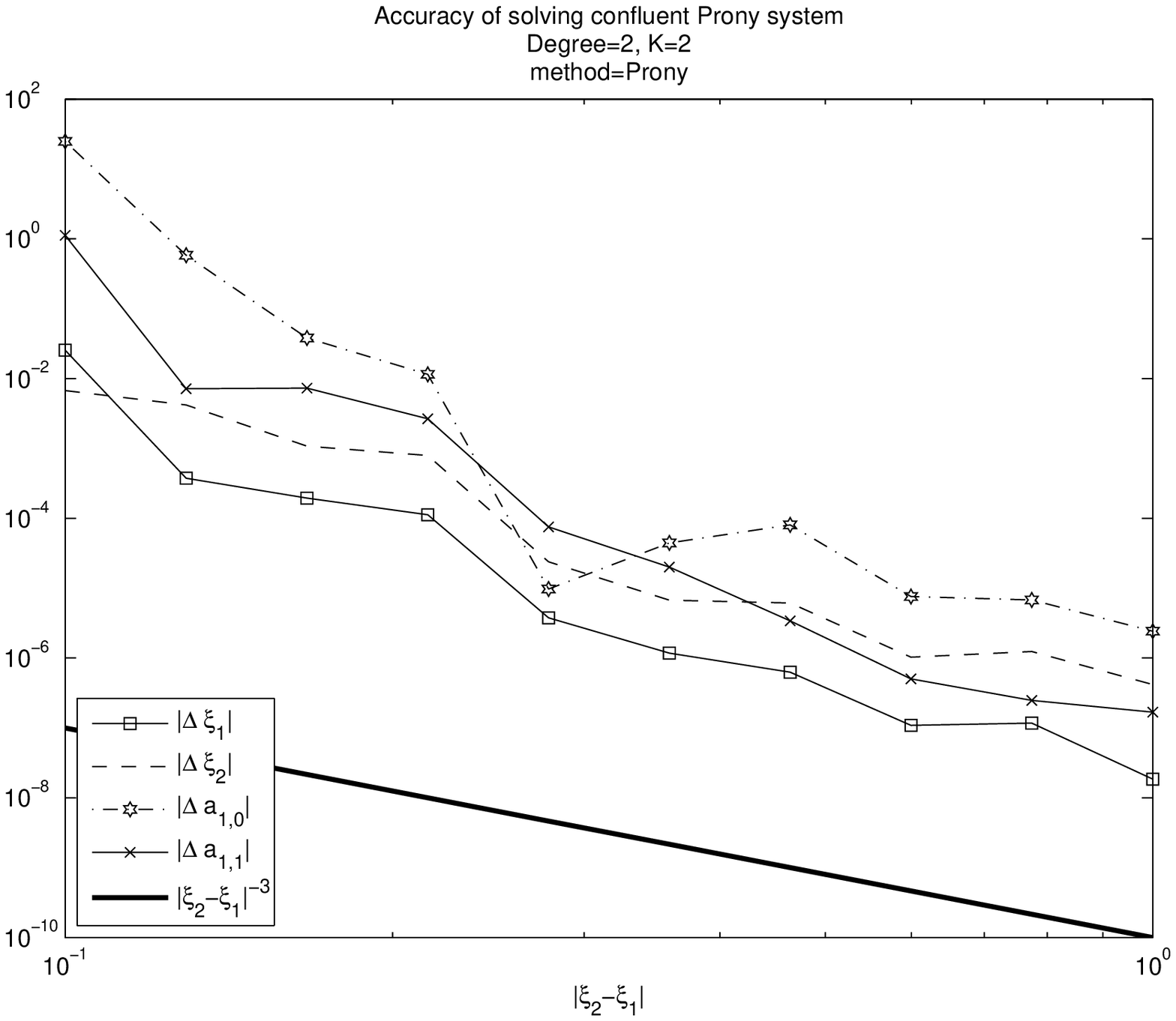}

} \subfloat[ESPRIT]{\includegraphics[width=0.33\columnwidth]{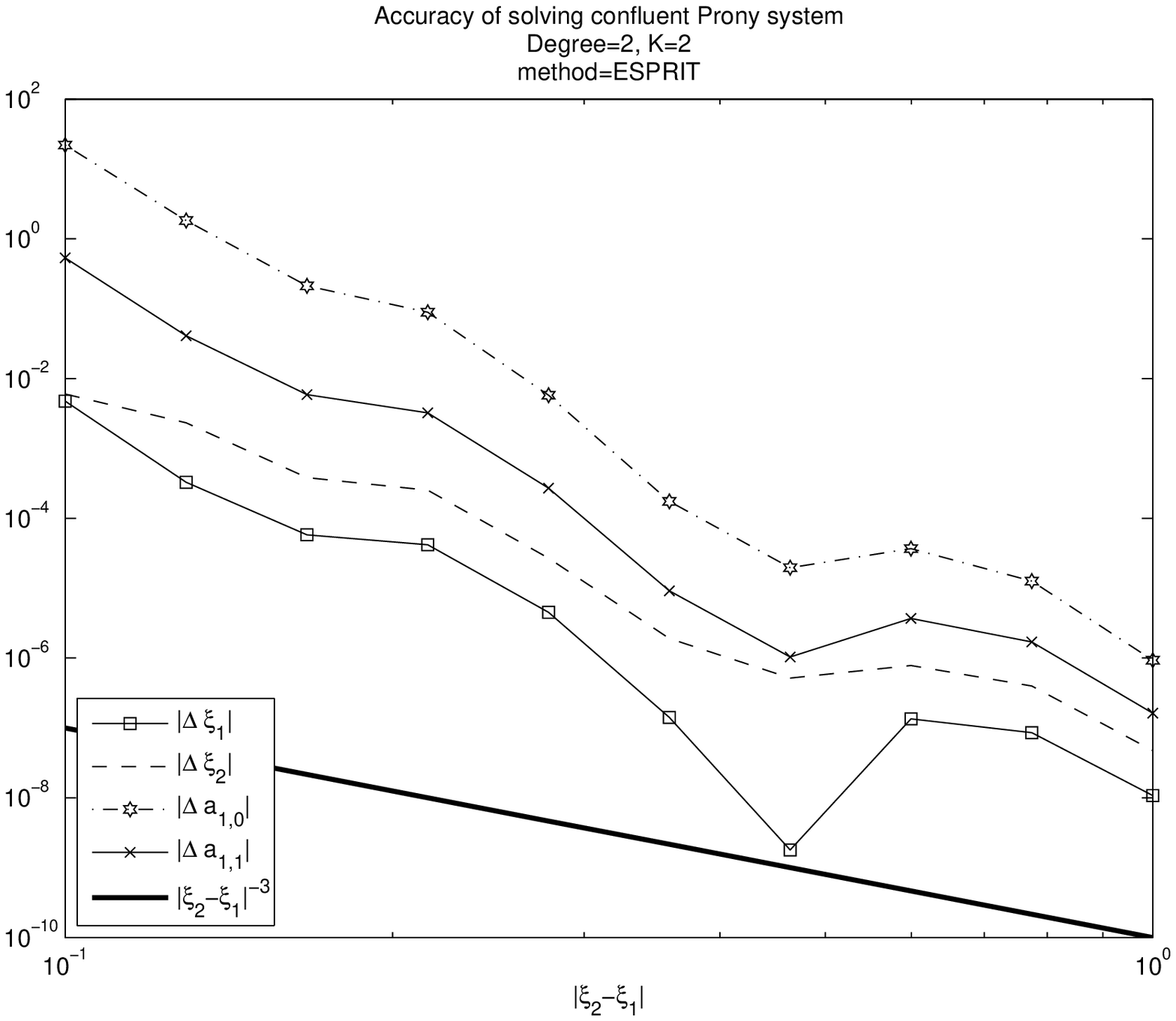}

}

\caption{Dependence of the reconstruction error on the node separation.}

\label{fig:separation}
\end{figure}

\subsubsection{Dependence on the node separation}

Finally, we checked the dependence of the reconstruction error on
the distance between the two nodes $\left|\jp_{2}-\jp_{1}\right|$.
The results are presented in \prettyref{fig:separation}. For all
the three methods, the results are consistent with
\[
\left|\Delta\jp_{i}\right|,\left|\Delta\jc_{i,j}\right|\sim\left|\jp_{2}-\jp_{1}\right|^{-D}.
\]

\subsection{Conclusions}

In the numerical experiments we have investigated the ``best possible
local accuracy'' via the least squares method, comparing it both
with the theoretical results of \prettyref{thm:local-lipshitz-estimates}
and with the performance of two ``global'' solution techniques, namely
Prony and ESPRIT methods, for small perturbations (high SNR). Our
results suggest that:
\begin{enumerate}
\item The numerical behavior of the solution in the case of small data perturbations
indeed exhibits the patterns predicted by \prettyref{thm:local-lipshitz-estimates},
in particular the qualitative dependence of the reconstruction error
on the values of the parameters of the problem.
\item The Prony solution method largely fails to separate the parameters
which could be separated in theory. Furthermore, its performance actually
\emph{degrades} when the highest coefficient $\left|\jc_{i,l_{i}-1}\right|$
is increased. ESPRIT separates the parameters better than Prony, but
is still worse than optimal.
\item In terms of absolute reconstruction error, ESPRIT is better than Prony
but still worse than the optimal LS.
\item In terms of dependence of the reconstruction error on the model order
and the node separation, both Prony and ESPRIT behave close to the
predicted law, namely exponential increase in the order and polynomial
increase in the separation distance.
\end{enumerate}

\section{\label{sec:discussion}Discussion}

We believe that the analytically approach of this paper has the potential
to provide relatively complete answer to several important questions
related to stable solution of Prony-type systems, as briefly discussed
below.

The numerical experiments suggest that the least squares method approximates
the optimal ``local'' behavior very well. However, it is well-known
that a very accurate initial approximation is required in order to
find the global minima. It is customary to use one of the global solution
methods to obtain such an initial value. Further analysis of the Prony
sets $\man$ may provide explicit conditions for such an initialization
to be sufficiently close to the true solution.

The general case $\nmeas\geq\nparams$ should be well-understood in
order to estimate the feasibility of taking more measurements than
strictly needed (oversampling). Without assumptions on the noise,
it is not a-priori obvious that averaging should improve the accuracy
in any way. Again, such an understanding is hopefully achievable via
the investigation of $\man$ with $\nmeas\gg\nparams$.

In practice it is often the case that neither the number of nodes
$\np$ nor the numbers $\left\{ l_{i}\right\} $ are known a-priori,
but only \emph{their upper }bounds\emph{. }In this case, given a noisy
measurement vector, more than one ``explanation'' is possible for
this data, in which case a good reconstruction algorithm needs somehow
to select the ``optimal'' configuration. One possible way to achieve
this goal is to characterize, for each configuration of the system
(i.e. $\left\{ \np,\left\{ l_{i}\right\} _{i=1}^{\np}\right\} $),
the ``stable regions'' of the corresponding measurement sets $\man$,
for which the accuracy function $\accr$ does not exceed a predefined
upper bound. Based on the initial measurement $\noims\in\complexfield^{\nmeas}$
and the error bound $\varepsilon$, an algorithm would choose the
closest ``stable measurement set'', i.e. select a configuration for
which the local accuracy is optimal. Using this approach, collision
of two nodes $\jp_{i}$, $\jp_{j}$ can in principle be handled in
a stable way by substituting the configuration $\left\{ \np,\left\{ l_{i}\right\} _{i=1}^{\np}\right\} $
with $\left\{ \np-1,\left\{ l_{1},\dots,l_{i}+l_{j},\dots,l_{\np}\right\} \right\} $
once the measurement vector leaves the stability region associated
with the former configuration. In this regard, we note that such a
singular behavior has been studied in \cite{yom2009Singularities}
(see also \cite{osborne1975some}), where it is shown that if the
solution is represented in the basis of divided differences, then
the inverse operator \emph{is} uniformly bounded with respect to the
corresponding expansion coefficients. Analogous developments for extraction
of multiple roots of polynomials \cite{zeng2005cmr} might be very
relevant as well.

In order to achieve the above goals, we propose to compute the function
$\accr$ as accurately as possible. For that purpose, more detailed
analysis of the Prony map%
\footnote{Its non-confluent version appears in a paper by Arnol'd \cite{arnold1986hpa}
under the name ``Vandermonde map''.%
} is necessary. In particular, its essential nonlinearity should be
quantified using the second-order terms in the Taylor expansion.

In addition to \eqref{eq:gen-prony}, other generalizations of the
basic Prony system \eqref{eq:basic-prony} appear in applications.
One such extension arises in Eckhoff's method \cite{eckhoff1995arf}
for reconstructing piecewise smooth functions from Fourier coefficients.
There, an additional parameter appears: namely, the measurements $m_{k}$
are given starting from some large index $k=M$. In \cite{batyomAlgFourier},
we have presented an algorithm for solving this system with high accuracy
(in the sense of asymptotic rate of convergence as $M\to\infty$.)
However, the question of ``maximal possible accuracy'' for this problem
is still open. It will be most desirable to reinterpret those results
in the sense of global stability bounds for Prony-like systems.

\bibliographystyle{siam}
\bibliography{../../bibliography/all-bib}

\end{document}